\documentclass{amsproc}

\usepackage{amssymb}

\usepackage{graphicx}

\usepackage[cmtip,all]{xy}

\usepackage[all]{xy}
\usepackage{amsthm}
\usepackage{amscd}
\usepackage{amsmath}
\usepackage{amsfonts}
\usepackage{amssymb}
\usepackage{stmaryrd}
\usepackage{graphicx}%
\usepackage{hyperref}
\hypersetup{
    bookmarks=true,         
    unicode=false,          
    pdftoolbar=true,        
    pdfmenubar=true,        
    pdffitwindow=false,     
    pdfstartview={FitH},    
    pdftitle={My title},    
    pdfauthor={Author},     
    pdfsubject={Subject},   
    pdfcreator={Creator},   
    pdfproducer={Producer}, 
    pdfkeywords={keyword1, key2, key3}, 
    pdfnewwindow=true,      
    colorlinks=false,       
    linkcolor=red,          
    citecolor=green,        
    filecolor=magenta,      
    urlcolor=cyan           
}

\usepackage{amscd,keyval}
\usepackage{setspace}
\makeatletter
\define@key{modCD}{cols}{\setlength{\minCDarrowwidth}{#1}}
\define@key{modCD}{rows}{\setlength{\modCD@rowsep}{#1}}
\newlength{\modCD@rowsep}

\renewenvironment{CD}[1][]
{\modCD@rowsep=20\ex@ 
  \setkeys{modCD}{#1}%
  \CDat
  \bgroup\relax\let\ampersand@&\iffalse}\fi
  \CD@true\vcenter\bgroup\let\\\math@cr\restore@math@cr\default@tag
  \tabskip\z@skip\baselineskip20\ex@
  &\hfill$\m@th##$\hfill\crcr}
{\crcr\egroup\egroup\egroup}
\atdef@ V#1V#2V{\CD@check{V..V..V}{%
  \llap{$\m@th\vcenter{\hbox{$\scriptstyle#1$}}$}%
  \left\downarrow\vbox to.5\modCD@rowsep{}\right.\kern-\nulldelimiterspace
  \rlap{$\m@th\vcenter{\hbox{$\scriptstyle#2$}}$}&&}}
\atdef@ A#1A#2A{\CD@check{A..A..A}{%
  \llap{$\m@th\vcenter{\hbox{$\scriptstyle#1$}}$}%
  \left\uparrow\vbox to.5\modCD@rowsep{}\right.\kern-\nulldelimiterspace
  \rlap{$\m@th\vcenter{\hbox{$\scriptstyle#2$}}$}&&}}
\makeatother

\newtheorem{theorem}{Theorem}[section]
\newtheorem{lemma}[theorem]{Lemma}

\theoremstyle{definition}
\newtheorem{definition}[theorem]{Definition}
\newtheorem{example}[theorem]{Example}
\newtheorem{proposition}[theorem]{Proposition}

\theoremstyle{remark}
\newtheorem{remark}[theorem]{Remark}

\numberwithin{equation}{section}

\begin{document}

\title{Quasi-Elliptic Cohomology I.}



\author{Zhen Huan}

\address{Zhen Huan, Department of Mathematics,
Sun Yat-sen University, Guangzhou, 510275 China} \curraddr{} \email{huanzhen@mail.sysu.edu.cn}
\thanks{The author was partially supported by NSF grant DMS-1406121.}


\subjclass[2010]{Primary 55}

\date{April, 2018}

\begin{abstract}
Quasi-elliptic cohomology is a variant of
elliptic cohomology theories. It is the orbifold K-theory of a space of constant 
loops. For global quotient orbifolds, it can be expressed in terms
of equivariant K-theories. Thus, the constructions on it can be made in a neat way. This theory reflects the geometric nature of the Tate curve.
In this paper we provide a systematic introduction
of its construction and definition.
\end{abstract}

\maketitle 

\section{Introduction}

An elliptic cohomology theory is an even periodic multiplicative
generalized cohomology theory whose associated formal group is the
formal completion of an elliptic curve. The elliptic cohomology theories form a sheaf of cohomology theories over the moduli stack of elliptic curves
$\mathcal{M}_{ell}$. Tate K-theory over Spec$\mathbb{Z}((q))$ is obtained when we restrict it to a punctured completed neighborhood of the cusp at $\infty$, i.e. the Tate
curve $Tate(q)$ over Spec$\mathbb{Z}((q))$  [Section 2.6,
\cite{AHS}]. The relation between Tate K-theory and string theory is better
understood than most known elliptic cohomology theories. In addition, Tate K-theory has the closest ties to Witten's original insight that the elliptic cohomology of  a space $X$ is related to the
$\mathbb{T}-$equivariant K-theory of the free loop space
$LX=\mathbb{C}^{\infty}(S^1, X)$ with the circle $\mathbb{T}$
acting on $LX$ by rotating loops. Ganter gave a careful interpretation in Section 2, \cite{Gan07} of this statement that the
definition of $G-$equivariant Tate K-theory for finite groups $G$
is modelled on the loop space of a global quotient orbifold.

Other than the theory over Spec$\mathbb{Z}((q))$, we can define variants of Tate K-theory over Spec$\mathbb{Z}[q]$ and Spec$\mathbb{Z}[q^{\pm}]$ respectively. The theory over Spec$\mathbb{Z}[q^{\pm}]$
is of especial interest.
Inverting $q$ allows us to define a sufficiently non-naive equivariant cohomology theory and to interpret some constructions more easily in terms of extensions of groups over
the circle. The resulting cohomology theory is called quasi-elliptic cohomology.
Its relation with Tate K-theory is \begin{equation}QEll^*_G(X)\otimes_{\mathbb{Z}[q^{\pm}]}\mathbb{Z}((q))=(K^*_{Tate})_G(X)
\label{tateqellequiv}\end{equation} which also reflects the geometric nature of the Tate curve. As discussed in Remark \ref{KMgroup},
$QEll^*_{\mathbb{T}}(\mbox{pt})$ has a direct interpretation in terms of the Katz-Mazur group scheme $T$ [Section 8.7, \cite{KM85}].
The idea of quasi-elliptic cohomology is  motivated by Ganter's
construction of Tate K-theory  \cite{Dev96}.
It is not an elliptic cohomology  but a more robust and algebraically simpler treatment of
Tate K-theory. This new theory can be interpreted in a neat form
by equivariant K-theories. Some
formulations in it can be generalized to equivariant cohomology
theories other than Tate K-theory.

Via quasi-elliptic cohomology theory, we show in this paper that $G-$equivariant Tate K-theory for any compact Lie group $G$
is given by the $\mathbb{T}-$equivariant $K-$theory of the ghost loops [Section \ref{ghost}], or constant loops [Section \ref{orbifoldloop}] inside the free loop space $LX$. Moreover, as shown in Section \ref{busb}, quasi-elliptic cohomology can be defined not only for $G-$spaces but also for orbifolds. Applying the same idea,
we obtain a loop construction for orbifold Tate K-theory via  orbifold quasi-elliptic cohomology theory.

This paper aims to provide a reference for this elegant theory and a systematic introduction of its construction and definition.
In Section \ref{loopmodel}, for any
compact Lie group $G$, we construct $G-$equivariant quasi-elliptic cohomology from a loop space via bibundles. Thus, we in fact give a construction by loop space of $G-$equivariant Tate K-theory for compact Lie groups $G$. In Section 2 \cite{HuanPower} we showed the construction when $G$ is a finite group, which, as shown in Section \ref{loopmodel},
can be generalized to the case when $G$ is a compact Lie group. We discuss the subtle points of this generalization in Section \ref{orbifoldloop}.
In Section \ref{conqec}
we give the definition of quasi-elliptic cohomology $QEll_G^*(-)$ with $G$ a compact Lie group, set up the theory and show its properties. We gave a different definition
of $QEll_G^*(-)$ with $G$ a finite group in Definition 3.10, \cite{HuanPower}, which is equivalent to the definition in this paper.  In Section \ref{orbifoldquasibeforepower}, we present the construction of orbifold
quasi-elliptic cohomology via the loop space of bibundles. Moreover, we give another construction motivated by Ganter's construction of orbifold Tate K-theory in \cite{Gan13}. The two constructions of orbifold
quasi-elliptic cohomology are equivalent.

In addition, I would like to introduce several other research progresses and the contribution of quasi-elliptic cohomology to the study of Tate K-theory and Tate curve.

Morava E-theories have many properties that reflects other homotopy theories. They serve as motivating examples for the research
on other cohomology theories. A classification of the level structure of its formal group is given in \cite{Isogeny}.
Strickland proved in \cite{Str98} that the Morava
$E-$theory of the symmetric group $\Sigma_n$ modulo a certain
transfer ideal classifies the power subgroups of rank $n$ of its
formal group. Stapleton proved in \cite{SSST} this result for generalized Morava E-theory via transchromatic character theory \cite{STGCM} \cite{STTCM}.
In each case, the power operation serves as a bridge connecting the homotopy theory and its formal group.
It is conjectured that we have classification theorems of the geometric structures of each elliptic curve in the same form.

In \cite{HuanPower} we construct a power operation of quasi-elliptic cohomology via explicit formulas that interwine
the power operation in K-theory and natural symmetries of the free
loop space. It is closely related to the stringy power operation
of Tate K-theory in \cite{Gan07}. One advantage of it over the latter operation is that its construction can be generalized to a family of other equivariant
cohomology theories.
Via it we show in \cite{HuanPower}   that the Tate K-theory of symmetric groups modulo a certain transfer
ideal, $K_{Tate}(\mbox{pt}/\!\!/\Sigma_N)/I^{Tate}_{tr}$, classifies finite subgroups of the Tate curve.
Applying the same idea and method,  we prove that, for the Tate K-theory of any finite abelian group $A$ modulo a certain transfer ideal,
$K_{Tate}(\mbox{pt}/\!\!/ A)/I_{tr}^A$, classifies $A-$Level structure of the Tate curve. This result will appear in a coming paper.

Ginzburg, Kapranov and Vasserot gave a definition of equivariant elliptic cohomology in \cite{GKV} and conjectured that any
elliptic curve $A$ gives rise to a unique equivariant elliptic
cohomology theory, natural in $A$. In his thesis \cite{Gepnerthesis}, Gepner presented a
construction of the equivariant elliptic cohomology that satisfies
a derived version of the Ginzburg-Kapranov-Vasserot axioms. We are interested in
how to give an explicit construction of an
orthogonal $G-$spectrum of quasi-elliptic cohomology and Tate K-theory.
In \cite{HuanSpec} we formulate a new category of orthogonal $G-$spectra and construct explicitly an orthogonal $G-$spectrum of quasi-elliptic cohomology in it. The
idea of the  construction can be applied to a family of equivariant cohomology theories, including Tate K-theory and generalized
Morava E-theories. Moreover, this construction provides a functor from the category
of global spectra to the category of orthogonal $G-$spectra.

The idea of global orthogonal spectra was first inspired in
Greenlees and May \cite{GP}.  Many classical theories, equivariant
stable homotopy theory,  equivariant bordism, equivariant
K-theory, etc, naturally exist not only for a single group but a
specific family of groups in a uniform way. Several models of
global homotopy theories have been established, including that by
Schwede \cite{SS}, Gepner \cite{Gepnerthesis} and Bohmann \cite{BG}.
In a conversation, Ganter indicated
that quasi-elliptic cohomology has better chances than Grojnowski equivariant elliptic
cohomology theory to be put together naturally in a uniform way
and made into an ultra-commutative global cohomology theory in the
sense of Schwede \cite{SS}.

We are establishing in a coming paper a more flexible global homotopy theory that is equivalent to Schwede's global homotopy theory.
Quasi-elliptic chomomology, Tate K-theory and generalized Morava E-theories can fit into the new global theory. We are still working on how
effective it is to judge whether a cohomology theory, especially an
elliptic cohomology theory, can be globalized.
The idea of the construction of the new global homotopy theory has been partially shown in Chapter 6 and 7 of the author's PhD thesis \cite{Huanthesis}. In Theorem 7.2.3 \cite{Huanthesis}
we show quasi-elliptic cohomology can be globalized in the new theory.

\subsection{Acknowledgement}
I would like to thank Charles Rezk. He is the one who set up the theory and a very inspiring advisor.
The orginal references for this topic are his unpublished manuscripts \cite{Rez11} and \cite{Rez16}. Most of this work was directed
by him. I would like to thank Nora Ganter who motivated quasi-elliptic cohomology theory.
I would also like many others for sharing their thoughts with me on my research
on quasi-elliptic cohomology, including Matthew Ando, Randy McCarthy, Martin Frankland, David Gepner, Stefan Schwede, Nathaniel Stapleton, Vesna Stojanoska. I would like to thank those friends spending time reading this topic and discussing with me, including Meng Guo, Matthew James Spong, Guozhen Wang, Chenchang Zhu.
Lastly, I would like to thank the referee for helpful comments on the paper.

\section{Models for loop spaces}\label{loopmodel}

To understand $QEll^*_G(X)$, it is essential to understand the
orbifold loop space. In this section, we will describe several
models for the loop space of $X/\!\!/G$. Lerman
discussed thoroughly in Section 3, \cite{LerStack} that the strict
2-category of Lie groupoids can be embedded into a weak 2-category
whose objects are Lie groupoids, 1-morphisms are bibundles and
2-morphisms equivariant diffeomorphisms between bibundles. Thus,
the free loop space of an orbifold $M$ is 
the category of bibundles from the trivial groupoid
$S^1/\!\!/\ast$ to the Lie groupoid $M$. In Definition
\ref{loopspacemorphism} we discuss $Loop_1(X/\!\!/G)$ and
introduce another model $Loop_2(X/\!\!/G)$ in Definition
\ref{loopspace3}.

The groupoid structure of $Loop_1(X/\!\!/G)$ generalizes $Map(S^1,
X)/\!\!/G$, which is a subgroupoid of it. Other than the
$G-$action, we also
consider the rotation by the circle group $\mathbb{T}$ on the objects 
and form the  groupoids $Loop_1^{ext}(X/\!\!/G)$ and $Loop_2^{ext}(X/\!\!/G)$. The later one contains all the information
of $Loop_1^{ext}(X/\!\!/G)$.


We also construct a loop space  $L_{orb}
(X/\!\!/G)$ by adding rotations to the orbifold loop space that Ganter used to define equivariant Tate K-theory in \cite{Gan07}.
It is a subgroupoid of $Loop_2^{ext}(X/\!\!/G)$.
The key groupoid $\Lambda(X/\!\!/G)$ in the construction of
quasi-elliptic cohomology is the full subgroupoid of $L_{orb}
(X/\!\!/G)$ consisting of the constant loops.  In order to unravel
the relevant notations in the construction of $QEll^*_G(X)$, we
study the orbifold loop space 
in Section \ref{orbifoldloop} and Section \ref{ghost}.

Moreover, we introduce the ghost loops $GhLoop(X/\!\!/G)$,  which is a subgroupoid of $Loop_1^{ext}(X/\!\!/G)$.
It is the third model of loop spaces from which we can construct quasi-elliptic cohomology. It has many good features that the other
three models don't have and is itself a model worth studying.



In Section \ref{orbifoldspre} we define $Loop_1(X/\!\!/G)$ and $Loop^{ext}_1(X/\!\!/G)$. In Section \ref{freereview} we recall the free loop space.
In
Section \ref{orbifoldloop} we interpret the enlarged groupoid
$Loop^{ext}_1(X/\!\!/G)$ and introduce the groupoid $\Lambda(X/\!\!/G)$ of constant loops, from which we construct quasi-elliptic cohomology.
In Section \ref{ghost} we present the
model of ghost loops.


\subsection{Bibundles}\label{orbifoldspre}

A standard reference for groupoids and bibundles is Section 2 and
3, \cite{LerStack}. For each pair of Lie groupoids $\mathbb{H}$
and $\mathbb{G}$,  the bibundles from $\mathbb{H}$ to
$\mathbb{G}$ are defined in Definition 3.25, \cite{LerStack}. The
category $Bibun(\mathbb{H}, \mathbb{G})$ has
bibundles from $\mathbb{H}$ to $\mathbb{G}$ as the objects and bundle maps as the morphisms.

The first question is how to define a "loop". Here we consider bibundles, i.e. the 1-morphisms in a weak 2-category
of Lie groupoids defined in Section 3, \cite{LerStack}

For any manifold $X$, let $Man_X$ denote the category of manifolds
over $X$, that is, the category whose objects are manifolds $Y$
equipped with a smooth map $Y\longrightarrow X$, and whose
morphisms are smooth maps $Y\longrightarrow Y'$ making the
following triangle commute. $$\xymatrix{Y\ar[r]\ar[d] &Y'\ar[dl]\\
X &}$$
A bibundle from $\mathbb{G}$ to $\mathbb{H}$ consists of a left principal $\mathbb{G}-$bundle $P$ over $H_0$ 
and a right action of $\mathbb{H}$
on $P$ via a $\mathbb{G}-$invariant map. 
The actions of $\mathbb{G}$ and $\mathbb{H}$ commute.
Below we give the definition of bibunldes, which unravels Definition 3.25, \cite{LerStack}.
\begin{definition}
Let $\mathbb{G}$ and $\mathbb{H}$ be Lie groupoids. A (left
principal) bibundle from $\mathbb{H}$ to $\mathbb{G}$ is a smooth
manifold $P$ together with \\1. A map $\tau: P\longrightarrow
\mathbb{G}_0$, and a surjective submersion $\sigma:
P\longrightarrow \mathbb{H}_0$. \\2. Action maps in
$Man_{G_0\times H_0}$ \begin{align*}
\mathbb{G}_1\,\strut_s\!\times\!\strut_{\tau}\, P &\longrightarrow
P \\ P\,\strut_{\sigma}\!\times\!\strut_t\, \mathbb{H}_1
&\longrightarrow P\end{align*} which we denote on elements as
$(g,p)\mapsto g\cdot p$ and $(p, h)\mapsto p\cdot h$, such that
\\1. $g_1\cdot (g_2\cdot p)= (g_1g_2)\cdot p$ for all $(g_1, g_2,
p)\in \mathbb{G}_1\,\strut_s\!\times\!\strut_t\,
\mathbb{G}_1\,\strut_s\!\times\!\strut_{\tau}\,  P$;
\\2. $(p\cdot h_1)\cdot h_2= p\cdot (h_1 h_2)$ for all $(p, h_1, h_2)\in P\,\strut_{\sigma}\!\times\!\strut_t\,  \mathbb{H}_1 \,\strut_s\!\times\!\strut_t\,
\mathbb{H}_1$; \\ 3. $p\cdot u_H(\sigma(p))=p$ and
$u_G(\tau(p))\cdot p = p$ for all $p\in P$. \\ 4. $g\cdot (p\cdot
h)= (g\cdot p)\cdot h$ for all $(g, p, h)\in
\mathbb{G}_1\,\strut_s\!\times\!\strut_{\tau}\,
P\,\strut_{\sigma}\!\times\!\strut_t\, \mathbb{H}_1$.
\\5. The map \begin{align*}\mathbb{G}_1\,\strut_s\!\times\!\strut_{\tau}\,
P&\longrightarrow P\,\strut_{\sigma}\!\times\!\strut_{\sigma}\,  P
\\ (g, p)&\mapsto (g\cdot p, p)\end{align*} is an isomorphism.

\label{bibundle}

\end{definition}

\begin{definition}
A bibundle map is a map $P\longrightarrow P'$ over
$\mathbb{H}_0\times \mathbb{G}_0$ which commutes with the
$\mathbb{G}-$ and $\mathbb{H}-$actions, i.e. the following
diagrams commute.
$$\begin{CD}\mathbb{G}_1\,\strut_s\!\times\!\strut_{\tau}\,  P
@>>>
P  \\ @VVV @VVV \\
\mathbb{G}_1\,\strut_s\!\times\!\strut_{\tau}\, P' @>>> P'
\end{CD}  \quad \begin{CD} P \,\strut_{\sigma}\!\times\!\strut_{t}\, \mathbb{H}_1 @>>> P \\ @VVV @VVV\\
P' \,\strut_{\sigma}\!\times\!\strut_{t}\,
\mathbb{H}_1 @>>>P'
\end{CD}$$

\label{bibundlemap}
\end{definition}

For each pair of Lie groupoids $\mathbb{H}$ and $\mathbb{G}$,  we
have a category $Bibun(\mathbb{H}, \mathbb{G})$ with as objects
bibundles from $\mathbb{H}$ to $\mathbb{G}$ and as morphisms the
bundle maps. 
The category of smooth functors from $\mathbb{H}$ to $\mathbb{G}$
is a subcategory of $Bibun(\mathbb{H}, \mathbb{G})$.

\begin{example}[$Bibun(S^1/\!\!/\ast, \ast/\!\!/G)$] According to the definition, a bibundle from $S^1/\!\!/\ast$
to $\ast/\!\!/G$ with $G$ a Lie group is a smooth manifold $P$
together with two maps $\pi: P\longrightarrow S^1$ a smooth
principal $G-$bundle and the constant map $r: P\longrightarrow
\ast$.  So a bibundle in this case is equivalent to a smooth
principal $G-$bundle over $S^1$. The morphisms in
$Bibun(S^1/\!\!/\ast, \ast/\!\!/G)$ are bundle isomorphisms.
\label{babyloop1}
\end{example}

\begin{definition}[$Loop_1(X/\!\!/G)$]

Let $G$ be a Lie group acting smoothly on a manifold $X$. We use
$Loop_1(X/\!\!/G)$ to denote the category $Bibun(S^1/\!\!/\ast,
X/\!\!/G)$, which generalizes Example \ref{babyloop1}. Each object
consists of a smooth manifold $P$ and two structure maps
$P\buildrel{\pi}\over\longrightarrow S^1$ a smooth principal
$G-$bundle  and $f: P\longrightarrow X$ a $G-$equivariant map. We
use the same symbol $P$ to denote both the object and the smooth
manifold when there is no confusion. A morphism 
is a $G-$bundle map $\alpha: P\longrightarrow P'$ making the
diagram below commute.
$$\xymatrix{S^1
&P \ar[l]_{\pi}\ar[d]^{\alpha}\ar[r]^{f} & X \\
&P' \ar[lu]^{\pi'}\ar[ru]_{f'} &}$$ Thus, the morphisms in
$Loop_1(X/\!\!/G)$ from $P$ to $P'$ are bundle isomorphisms.



\label{loopspacemorphism}\end{definition}

Only the $G-$action on $X$ is considered in $Loop_1(X/\!\!/G)$. We
add the rotations by adding more morphisms into the groupoid.

\begin{definition}[$Loop^{ext}_1(X/\!\!/G)$]\label{loopext3space} 
Let $Loop^{ext}_1(X/\!\!/G)$ denote the groupoid with the same
objects as $Loop_1(X/\!\!/G)$. Each morphism
consists of the pair $(t, \alpha)$ where $t\in\mathbb{T}$ is a
rotation and $\alpha$ is a $G-$bundle map. They make the diagram
below commute.
$$\xymatrix{S^1\ar[d]_{t}
&P \ar[l]_{\pi}\ar[d]_{\alpha}\ar[r]^{f} & X \\S^1 &P'
\ar[l]^{\pi'}\ar[ru]_{f'} &}$$

The groupoid $Loop_1(X/\!\!/G)$ is a subgroupoid of
$Loop^{ext}_1(X/\!\!/G)$.
\end{definition}

\subsection{Free loop space} \label{freereview}

In this section we recall the free loop space of a $G-$space and discuss the actions on it by $Aut(S^1)$
and the loop group $LG$. We will also show its relation with physics.

For any space $X$, we have the free loop space of $X$
\begin{equation}LX:=\mathbb{C}^{\infty}(S^1, X). \end{equation} It
comes with an evident action by the circle group
$\mathbb{T}=\mathbb{R}/\mathbb{Z}$ defined by rotating the circle
\begin{equation}t\cdot \gamma:= (s\mapsto \gamma (s+t)), \mbox{
}t\in S^1, \mbox{   } \gamma\in LX.\end{equation}

Let $G$ be a compact Lie group. Suppose $X$ is a right
$G$-space. The free loop space $LX$  is  equipped with an action
by the loop group $LG$   \begin{equation}\delta\cdot
\gamma:=(s\mapsto \delta(s)\cdot \gamma(s)),\mbox{ for any } s\in
S^1, \mbox{   }\delta\in LX, \mbox{ }\mbox{
    }\gamma\in LG.\end{equation}

Combining the action by group of automorphisms $Aut(S^1)$ on the
circle and the action by $LG$, we get an action by the extended
loop group $\Lambda G$ on $LX$. $\Lambda G:=LG\rtimes\mathbb{T}$
is a subgroup of
\begin{equation} LG\rtimes Aut(S^1), \mbox{             }(\gamma, \phi)\cdot (\gamma', \phi'):= (s\mapsto \gamma(s) \gamma'(\phi^{-1}(s)), \phi\circ\phi')\end{equation}
with $\mathbb{T}$ identified with the group of rotations on $S^1$.
$\Lambda G$ acts on $LX$ by
\begin{equation}\delta \cdot(\gamma, \phi):= (t\mapsto
\delta(\phi(t))\cdot\gamma(\phi(t))), \mbox{  for any }(\gamma,
\phi)\in \Lambda G,\mbox { and    }\delta \in
LX.\label{loop2action}\end{equation} It's straightforward to check
(\ref{loop2action}) is a well-defined group action.

Let $G^{tors}$ denote the set of torsion elements in $G$.  Let $g\in G^{tors}$. 
Define $L_{g}G$ to be the twisted loop group
\begin{equation}
\{\gamma: \mathbb{R}\longrightarrow
G|\gamma(s+1)=g^{-1}\gamma(s)g\}.
\end{equation}
The multiplication of it is defined by
\begin{equation}(\delta\cdot\delta')(t)=\delta(t)\delta'(t)\mbox{,     for      any     }\delta, \delta'\in
L_{g}G,  \mbox{    and     }t\in\mathbb{R}.\end{equation} The
identity element $e$ is the constant map sending all the real
numbers to the identity element of $G$. Similar to $\Lambda G$, we
can define $L_{g}G\rtimes\mathbb{T}$ whose multiplication is
define by
\begin{equation} (\gamma, t)\cdot (\gamma', t'):= (s\mapsto \gamma(s) \gamma'(s+t), t+t').\label{lkgtmulti}\end{equation}

The set of constant maps $\mathbb{R}\longrightarrow G$ in
$L_{g}G$ is a subgroup of it, i.e. the centralizer $C_G(g)$.

\bigskip

Before we go on, we discuss the physical meaning of the twisted loop group. Recall that the gauge group of a
principal bundle is defined to be the group of its vertical
automorphisms. The readers may refer \cite{MV} for more details on
gauge groups. For a $G-$bundle $P\longrightarrow S^1$, let $L_P G$
denote its gauge group.

We have the well-known facts below. 

\begin{lemma}The principal $G-$bundles over $S^1$ are classified up to isomorphism by
homotopy classes $$[S^1, BG]\cong \pi_0G/\mbox{conj}.$$ Up to
isomorphism every principal $G-$bundle over $S^1$ is isomorphic to
one of the forms $P_{\sigma}\longrightarrow S^1$ with $\sigma\in
G$ and
$$P_{\sigma}:=\mathbb{R}\times G/ (s+1, g)\sim(s, \sigma g).$$ A complete collection of isomorphism classes is given by  a
choice of representatives for each conjugacy class of
$\pi_0G$.\label{ghostlem1}\end{lemma}

For the gauge group $L_{P_{\sigma}} G$ of the bundle $P_{\sigma}$,
we have the conclusion.

\begin{lemma}For the bundle
$P_{\sigma}\longrightarrow S^1$, $L_{P_{\sigma}}G$ is isomorphic
to the twisted loop group $L_{\sigma}G$.
\label{ghostlem2}\end{lemma}

\begin{proof}


Each automorphism $f$ of an object
$S^1\buildrel{\pi}\over\leftarrow
P_{\sigma}\buildrel{\widetilde{\delta}}\over\rightarrow X$ in
$Loop^{ext}_1(X/\!\!/G)$ has the form
\begin{equation}\begin{CD}P' @>{[s, g]\mapsto [s, \gamma_f(s)g]}>> P \\ @VVV @VVV \\ S^1 @>{=}>> S^1
\end{CD}\end{equation} for some $\gamma_f: \mathbb{R}\longrightarrow
G$.  The morphism is well-defined if and only if
$\gamma_f(s+1)=\sigma^{-1}\gamma_f(s)\sigma$.

So we get a well-defined map $$F: L_{P_{\sigma}}G\longrightarrow
L_{\sigma}G\mbox{,    } f\mapsto\gamma_f.$$  It's a bijection.
Moreover, by the property of group action, $F$ sends the identity
map to the constant map $\mathbb{R}\longrightarrow G\mbox{,   }
s\mapsto e$, which is the trivial element in $L_{\sigma}G$, and
for two automorphisms $f_1$ and $f_2$ at the object, $F(f_1\circ
f_2)= \gamma_{f_1}\cdot \gamma_{f_2}$. So $L_{P_{\sigma}}G$ is
isomorphic to $ L_{\sigma}G$.

\end{proof}

\subsection{Orbifold Loop Space}\label{orbifoldloop}

In this section, we present the loop space $Loop_2(X/\!\!/G)$. Based on these models, we construct the groupoid $Loop^{ext}_2(X/\!\!/G)$ and show its relation to
$Loop^{ext}_1(X/\!\!/G)$, the model by bibundles. 
These models, however, are not good enough to study. Instead, 
we consider a subgroupoid $\Lambda(X/\!\!/G)$ of $Loop^{ext}_2(X/\!\!/G)$ consisting of constant loops. It can be constructed from the orbifold loop space in Section 2.1 \cite{Gan07} that
Ganter used to formulate Tate K-theory and show
its relation with loop spaces.

Let $G$ be a Lie group acting smoothly on a manifold $X$.

\begin{definition}[$Loop_2(X/\!\!/G)$]\label{loopspace3}
Let
$Loop_2(X/\!\!/G)$ denote the groupoid whose objects are $(\sigma,
\gamma)$ with $\sigma\in G$ and $\gamma: \mathbb{R}\longrightarrow
X$
a continuous map such that $\gamma(s+1)= \gamma(s)\cdot\sigma$, for any $s\in\mathbb{R}$. 
A morphism $\alpha: (\sigma, \gamma)\longrightarrow (\sigma',
\gamma')$ is a continuous map $\alpha: \mathbb{R}\longrightarrow
G$ satisfying $\gamma'(s)= \gamma(s)\alpha(s)$. Note that
$\alpha(s)\sigma'=\sigma\alpha(s+1)$, for any $s\in\mathbb{R}$.
\end{definition}
The objects of $Loop_2(X/\!\!/G)$ can be identified
with the space
$$\coprod\limits_{g\in G}\mathcal{L}{_g} X$$ where
\begin{equation}\mathcal{L}{_g}
X:=\mbox{Map}_{\mathbb{Z}}(\mathbb{R},
X).\label{chongL}\end{equation}

In each $\mathcal{L}{_g}
X$, the group $\mathbb{Z}$ acts on $\mathbb{R}$ by
group multiplication and the generator $1$ in $\mathbb{Z}$ acts on $X$ as the element $g$ via the $G-$action.
The groupoid $\mathcal{L}{_g}
X/\!\!/ L_gG$ is a full subgroupoid of $Loop_2(X/\!\!/G)$.

Now we consider the extended loop spaces  with richer morphism
spaces.
\begin{definition}[$Loop^{ext}_2(X/\!\!/G)$]\label{loopext3space}

Let $Loop^{ext}_2(X/\!\!/G)$ denote the groupoid with the same objects
as $Loop_2(X/\!\!/G)$. A morphism $$(\sigma, \gamma)\longrightarrow
(\sigma', \gamma')$$ consists of the pair $(\alpha, t)$ with
$\alpha:\mathbb{R}\longrightarrow G$ a continuous map and
$t\in\mathbb{R}$ a rotation on $S^1$ satisfying
$\gamma'(s)=\gamma(s-t)\alpha(s-t)$.

The groupoid $Loop_2(X/\!\!/G)$ is a subgroupoid of $Loop^{ext}_2(X/\!\!/G)$.

\end{definition}
\begin{lemma}The groupoid $Loop^{ext}_1(X/\!\!/G)$ is isomorphic to a full subgroupoid of $Loop^{ext}_2(X/\!\!/G)$.\end{lemma}

\begin{proof}
Define a functor  $$F: Loop^{ext}_1(X/\!\!/G)\longrightarrow
Loop^{ext}_2(X/\!\!/G)$$ by sending an object
$$\begin{CD} S^1 @<{\pi}<< P @>{f}>> X\end{CD}$$ to $(\sigma,
\gamma)$ with $\gamma(s):= f([s, e])$ and
$\sigma=\gamma(0)^{-1}\gamma(1)$ and sending a morphism
$$\xymatrix{S^1\ar[d]_{t} &P \ar[l]_{\pi}\ar[d]^F\ar[r]^{f} & X
\\ S^1 &P' \ar[l]^{\pi'}\ar[ru]_{f'} &}$$ to $(\alpha, t): (\sigma,
\gamma)\longrightarrow (\sigma', \gamma')$ with $\alpha(s):=F([s,
e])^{-1}.$

$F$ is a fully faithful functor but not essentially surjective.

\end{proof}

Therefore, $Loop^{ext}_2(X/\!\!/G)$ contains all the
information of $Loop^{ext}_1(X/\!\!/G)$. Next we show a
skeleton of the larger groupoid.


For each $g\in G$, $\mathcal{L}{_g} X/\!\!/
L_gG\rtimes\mathbb{T}$ is a full subgroupoid of $Loop^{ext}_2(X/\!\!/G)$ where $L_{g}G\rtimes\mathbb{T}$ acts on
$\mathcal{L}{_g} X$ by
\begin{equation}\delta \cdot(\gamma, t):= (s\mapsto
\delta(s+t)\cdot\gamma(s+t)), \mbox{  for any }(\gamma, t)\in
L^k_gG\rtimes\mathbb{T},\mbox { and    }\delta \in
\mathcal{L}{_g} X.\label{chongaction}\end{equation}
The action by $g$ on $\mathcal{L}{_g} X$ coincides with that
by $1\in\mathbb{R}$. So we have the isomorphism
\begin{equation}
L_gG\rtimes\mathbb{T}=L_g G\rtimes \mathbb{R}/\langle
(\overline{g}, -1)\rangle,\label{bengkui}
\end{equation} where $\overline{g}$ represents the constant loop
$\mathbb{T}\longrightarrow \{g\}\subseteq G$. 



We have already proved Proposition \ref{loop1equivske}.
\begin{proposition}
Let $G$ be a compact Lie group. The groupoid
$$\mathcal{L}(X/\!\!/G):=\coprod_{[g]}\mathcal{L}{_g}
X/\!\!/L_gG\rtimes\mathbb{T}$$ is a skeleton of
$Loop^{ext}_2(X/\!\!/G)$, where the coproduct goes over conjugacy
classes in $\pi_0G$.   \label{loop1equivske}
\end{proposition}

\begin{definition}[$Loop^{ext, tors}_2(X/\!\!/G)$]
Let $Loop^{ext, tors}_2(X/\!\!/G)$ denote the full subgroupoid of
$Loop^{ext}_2(X/\!\!/G)$ whose objects are the pairs $(\sigma,
\gamma)$ with $\sigma\in G^{tors}$ and $\gamma:
\mathbb{R}\longrightarrow X$ a continuous map such that
$\gamma(s+1)= \gamma(s)\cdot\sigma$, for any $s\in\mathbb{R}$.

\label{torsionlambda}
\end{definition}

The groupoid $Loop^{ext, tors}_2(X/\!\!/G)$ contains all the
information we want. But is it convenient enough to study? When
$G$ is not finite, The isotropy group $L_gG\rtimes\mathbb{T}$ of an object
in $\mathcal{L}{_g}X$ is an infinite dimensional topological
group. We need even smaller groups to define a good orbifold loop
space. Thus, we consider those elements $[\gamma, t]\in
\Lambda_{g}G$ with $\gamma$ a constant loop. They form a
subgroup of $\Lambda_{g}G$ which is the quotient group of
$C_G(g)\times \mathbb{R}/l\mathbb{Z}$ by the normal subgroup
generated by $(g, -1)$. We denote it by
$\Lambda_G(g)$.  
When $G$ is a compact Lie group, $\Lambda_G(g)$
is also a compact Lie group.

Therefore, instead of $Loop^{ext, tors}_2(X/\!\!/G)$, we
consider
a subgroupoid $L_{orb}(X/\!\!/G)$  of it in
Definition \ref{orbifoldloopspace}, which is closely related to Ganter's orbifold loop space in \cite{Gan07}, and afterwards
a full subgroupoid $\Lambda(X/\!\!/G)$ of $L_{orb}(X/\!\!/G)$, which we use to define quasi-elliptic cohomology in Section \ref{orbqec}.

We need Definition \ref{lambdak1} first.

\begin{definition}
Let $C_G(g, g')$ denote the set $$\{x\in G| gx=xg'\}.$$

Let $\Lambda_G(g, g')$ denote
the quotient of $C_G(g, g')\times \mathbb{R}/l\mathbb{Z}$ under
the equivalence $$(x, t)\sim (gx, t-1)=(xg',
t-1).$$\label{lambdak1}\end{definition}

\begin{definition} [$L_{orb} (X/\!\!/G)$ and $\Lambda(X/\!\!/G)$] \label{orbifoldloopspace}

Let $L_{orb}(X/\!\!/G)$  denote the groupoid with the same objects
as $Loop^{ext, tors}_2(X/\!\!/G)$, i.e. the space
$$\coprod\limits_{g\in G^{tors}}\mathcal{L}{_g} X,$$ and with
morphisms the space $\coprod\limits_{g, g'\in
G^{tors}}\Lambda_G(g,g')\times X^g$.

For $\delta\in \mathcal{L}{_g} X$, $[a, t]\in\Lambda_G(g,
g')$,
\begin{equation}\delta\cdot ([a, t], \delta) := (s\mapsto
\delta(s+t)\cdot a) \in \mathcal{L}{_{g'}}
X.\label{actlt}\end{equation} in the same way as
(\ref{chongaction}).

$L_{orb}(X/\!\!/G)$ has the same objects as the orbifold loop space in \cite{Gan07} and has more morphisms with the $\mathbb{T}-$action added.


The groupoid $\Lambda(X/\!\!/G)$ defined in Example \ref{torsionlambda}  is the
full subgroupoid of $L_{orb}(X/\!\!/G)$  with constant loops
as objects. In Section \ref{orbqec}, we have a thorough discussion of $\Lambda(X/\!\!/G)$.


\label{olpwelldefined}\end{definition}

Let $G^{tors}_{conj}$ denote a set of representatives of
$G-$conjugacy classes in $G^{tors}$.

\begin{lemma}The
groupoid $$\coprod_{g\in G^{tors}_{conj}}\mathcal{L}{_g}
X/\!\!/\Lambda_G(g),$$ is a skeleton of $L_{orb}
(X/\!\!/G)$. It does not depend
on the choice of representatives of the $G-$conjugacy classes. \label{orbskel}


\end{lemma}

The proof is analogous to that of Lemma \ref{loop1equivske}.

\subsection{Ghost Loops} \label{ghost}

Let $G$ be a compact Lie group and $X$ a $G-$space. In this
section we introduce a subgroupoid $GhLoop(X/\!\!/G)$ of
$Loop^{ext}_1(X/\!\!/G)$, which can be computed locally.

\begin{definition}[Ghost Loops]The groupoid of ghost loops is defined to be the full subgroupoid
$GhLoop(X/\!\!/G)$ of $Loop^{ext}_1(X/\!\!/G)$ consisting of
objects $S^1\leftarrow
P\buildrel{\widetilde{\delta}}\over\rightarrow X$ such that
$\widetilde{\delta}(P)\subseteq X$ is contained in a single
$G-$orbit. \label{ghostloopdef} \end{definition}

For a given $\sigma\in G$, define the space
\begin{equation}
GhLoop_{\sigma}(X/\!\!/G):=\{\delta\in \mathcal{L}{_{\sigma}}
X |\delta(\mathbb{R})\subseteq G\delta(0) \}.
\end{equation}

We have a corollary of Proposition \ref{loop1equivske} below.
\begin{proposition}$GhLoop(X/\!\!/G)$ is equivalent to the
groupoid
$$\Lambda(X/\!\!/G):=\coprod_{[\sigma]}GhLoop_{\sigma}(X/\!\!/G)/\!\!/L^1_{\sigma}G\rtimes\mathbb{T}$$
where the coproduct goes over conjugacy classes in $\pi_0G$.
\label{skullofghost}
\end{proposition}

\begin{example}
If $G$ is a finite group, it has the discrete topology. In this
case, $LG$ consists of constant loops and, thus,  is isomorphic to
$ G$. The space of objects of $GhLoop(X/\!\!/G)$ can be identified
with $X$. For $\sigma\in G$ and any integer $k$, $L_{\sigma} G$
can be identified with $C_G(\sigma)$;
$L_{\sigma}G\rtimes\mathbb{T}\cong C_G(\sigma)\times
\mathbb{R}/\langle (\sigma, -1)\rangle$; and
$GhLoop_{\sigma}(X/\!\!/G)$ can be identified with $X^{\sigma}$.

\label{finiteghost}
\end{example}

Unlike true loops, ghost loops have the property that they can be
computed locally, as shown in  the lemma below. The proof is left
to the readers.

\begin{proposition}If $X=U\cup V$ where $U$ and $V$ are $G-$invariant open subsets,
then $GhLoop(X/\!\!/G)$ is isomorphic to the fibred product of
groupoids $$GhLoop(U/\!\!/G)\cup_{GhLoop((U\cap
V)/\!\!/G)} GhLoop(V/\!\!/G).$$ \label{ghostmv} 
\end{proposition}
Thus, the ghost loop construction satisfies Mayer-Vietoris
property. Moreover, it has the change-of-group property.

\begin{proposition}Let $H$ be a closed subgroup of $G$. It acts on the
space of left cosets $G/H$ by left multiplication. Let $\mbox{pt}$
denote the single point space with the trivial $H-$action. Then we
have the equivalence  of topological groupoids between
$Loop^{ext}_1((G/H)/\!\!/G)$ and $Loop^{ext}_1(\mbox{pt}/\!\!/H)$.
Especially, there is an equivalence between the groupoids
$GhLoop((G/H)/\!\!/G)$ and $GhLoop(\mbox{pt}/\!\!/ H)$.
\label{globalghost}\end{proposition}

\begin{proof}

First we define a functor $F:
Loop^{ext}_1((G/H)/\!\!/G)\longrightarrow
Loop^{ext}_1(\mbox{pt}/\!\!/ H)$ sending an object $S^1\leftarrow
P\buildrel{\widetilde{\delta}}\over\rightarrow G/H$  to
$S^1\leftarrow Q\rightarrow \{eH\}=\mbox{pt}$ where
$Q\longrightarrow eH$ is the constant map, and $Q\longrightarrow
S^1$ is the pull back bundle $$\xymatrix{Q \ar[r] \ar[d] &\{eH\}
\ar@{^{(}->}[d]
\\ P\ar[r] &G/H.}$$

It sends a morphism $$\xymatrix{P'\ar[r] \ar[d]&P \ar[r] \ar[d]
&G/H\\ S^1\ar[r] &S^1&}$$ to the morphism
$$\xymatrix{Q' \ar[r] \ar[d] &Q\ar[r] \ar[d] &\{eH\} \ar[d] \\ P'\ar[r]\ar[d] &P\ar[r]\ar[d] &G/H\\
S^1 \ar[r] &S^1 & } $$ where all the squares are pull-back.

In addition, we can define a functor $F': Loop^{ext}_1
(\mbox{pt}/\!\!/ H) \longrightarrow Loop^{ext}_1((G/H)/\!\!/G)$
sending an object $S^1\leftarrow
Q\rightarrow \mbox{pt}$ to $S^1\leftarrow G\times_HQ\rightarrow
G\times_H\mbox{pt}=G/H$ and sending a morphism $$\xymatrix{Q'\ar[r] \ar[d] & Q \ar[d] \\ S^1\ar[r]
&S^1}$$ to $$\xymatrix{G\times_HQ'\ar[r] \ar[d] & G\times_HQ \ar[d] \ar[r] &G\times_H\mbox{pt}=G/H\\ S^1\ar[r]
&S^1 &}$$

$F\circ F'$ and $F'\circ F$ are both identity maps. So the
topological groupoids $Loop^{ext}_1((G/H)/\!\!/G)$ and
$Loop^{ext}_1(\mbox{pt}/\!\!/ H)$ are equivalent.

\bigskip

We can prove the equivalence between $GhLoop((G/H)/\!\!/G)$ and
$GhLoop(\mbox{pt}/\!\!/ H)$ in the same way.
\end{proof}


\begin{remark}In general, if $H^*$ is an equivariant cohomology theory,
Proposition \ref{globalghost} implies the functor $$X/\!\!/G
\mapsto H^*(GhLoop(X/\!\!/G))$$ gives a new equivariant cohomology
theory. When $H^*$ has the change of group isomorphism, so does
$H^*(GhLoop(-))$.
\end{remark}

\section{Quasi-elliptic cohomology $QEll^*_G(-)$}

\label{conqec}
In Section \ref{orbqec} we introduce the
construction of quasi-elliptic cohomology first in terms of
orbifold K-theory and then equivariant K-theory. We show the
properties of the theory in Section \ref{propertiesqec}.
The main
references for Section \ref{conqec} are Rezk's unpublished work \cite{Rez11} and the author's PhD thesis \cite{Huanthesis}.

Before that in
Section \ref{lambdarepresentationlemma} we show the representation
ring of \begin{equation}\Lambda_G(g):= C_G(g)\times
\mathbb{R}/\langle (g, -1)\rangle,\label{lambdadef}\end{equation}which is a factor of $QEll^*_G(\mbox{pt})$.

Moreover, in Section \ref{orbqec} we discuss the $\Lambda-$ring structure of $QEll^*_G(X)$.
In Section \ref{lambdarepresentationlemma} we introduce two groups $\Lambda_G^k(g)$ and $\Lambda_n(g)$ related closely to $\Lambda_G(g)$. We discuss some $\Lambda-$ring homomorphisms
between the representation rings of them, which are essential in the construction of $\Lambda-$ring homomorphisms on quasi-elliptic cohomology in Section \ref{orbqec}.
A good reference for $\Lambda-$rings is the book \cite{Yau10}.

\subsection{Preliminary: representation ring of
$\Lambda_G(g)$}\label{lambdarepresentationlemma}

Let $q: \mathbb{T}\longrightarrow U(1)$
be the isomorphism $t\mapsto e^{2\pi it}$. The representation ring
$R\mathbb{T}$ of the circle group is $\mathbb{Z}[q^{\pm}]$.

For any compact Lie group $G$ and a torsion element $g\in G$, we have an exact sequence
$$1\longrightarrow C_G(g)\longrightarrow
\Lambda_G(g)\buildrel{\pi}\over\longrightarrow\mathbb{T}\longrightarrow
0$$ where the first map is $g\mapsto [g, 0]$ and the second map is
\begin{equation}\pi([g, t])= e^{2\pi it}.\label{pizq}\end{equation}
The map $\pi^*: R\mathbb{T}\longrightarrow R\Lambda_G(g)$ equips
the representation ring $R\Lambda_G(g)$ the structure as an $R\mathbb{T}-$module.

There is a relation between the representation ring of $C_G(g)$
and that of $\Lambda_G(g)$, which is shown as Lemma 1.2
in \cite{Rez11} and Lemma 2.4.1 in \cite{Huanthesis}.
\begin{lemma}

$\pi^*: R\mathbb{T}\longrightarrow R\Lambda_G(g)$ exhibits
$R\Lambda_G(g)$ as a free $R\mathbb{T}-$module.

In particular, there is an $R\mathbb{T}-$basis of $R\Lambda_G(g)$
given by irreducible representations $\{V_{\lambda}\}$, such that
restriction $V_{\lambda}\mapsto V_{\lambda}|_{C_G(g)}$ to $C_G(g)$
defines a bijection between $\{V_{\lambda}\}$ and the set
$\{\lambda\}$ of irreducible representations of
$C_G(g)$.\label{cl}
\end{lemma}
\begin{proof}
Note that $\Lambda_G(g)$ is isomorphic to
$$C_G(g)\times\mathbb{R}/\langle(g, -1)\rangle.$$ Thus, it is the quotient of the product of two compact Lie groups.

Let $\lambda: C_G(g)\longrightarrow GL(n, \mathbb{C})$ be an
$n-$dimensional $C_G(g)-$representation with representation space
$V$ and $\eta: \mathbb{R}\longrightarrow GL(n, \mathbb{C})$ be a
representation of $\mathbb{R}$ such that $\lambda(g)$ acts on $V$
via scalar multiplication by $\eta(1)$. Define
\begin{equation}\lambda\odot_{\mathbb{C}} \eta([h, t]):
=\lambda(h)\eta(t). \label{lambdaeq}\end{equation} It's straightforward to verify
$\lambda\odot_{\mathbb{C}} \eta$ is a $n-$dimensional
$\Lambda_G(g)-$representation with representation space $V$.

Any irreducible $n-$dimensional representation of the quotient
group $\Lambda_G(g)=C_G(g)\times\mathbb{R}/\langle(g, -1)\rangle$
is an irreducible $n-$dimensional representation of the product
$C_G(g)\times\mathbb{R}/\langle(g, -1)\rangle$. And any finite
dimensional irreducible representation of the product of two
compact Lie groups is the tensor product of an irreducible
representation of each factor. So any irreducible representation
of the quotient group $\Lambda_G(g)$ is the tensor product of an
irreducible representation $\lambda$ of $C_G(g)$ with
representation space $V$ and an irreducible representation $\eta$
of $\mathbb{R}$. Any irreducible complex representation $\eta$ of
$\mathbb{R}$ is one dimensional. So the representation space of
$\lambda\otimes \eta$ is still $V$. Let $l$ be the order of $g$.
$\eta(1)^l=I$. We need $\eta(1)=\lambda (g)$. So
$\eta(1)=e^{\frac{2\pi ik}{l}}$ for some $k\in \mathbb{Z}$. So
$$\eta(t)= e^{\frac{2\pi i(k+lm)t}{l}}.$$ Any $m\in\mathbb{Z}$
gives a choice of  $\eta$ in this case. And $\eta$ is a
representation of $\mathbb{R}/l\mathbb{Z}\cong \mathbb{T}$.

Therefore, we have a bijective correspondence between

(1) isomorphism classes of irreducible
$\Lambda_G(g)-$representation $\rho$, and

(2) isomorphism classes of pairs $(\lambda, \eta)$ where $\lambda$
is an irreducible $C_G(g)-$representation and
$\eta:\mathbb{R}\longrightarrow \mathbb{C}^*$ is a character such
that $\lambda(g)=\eta(1)I$. $\lambda=\rho|_{C_G(g)}$.
\end{proof}

\begin{remark} We can make a canonical choice of $\mathbb{Z}[q^{\pm}]$-basis
for $R\Lambda_{G}(g)$. For each irreducible $G$-representation
$\rho: G\longrightarrow Aut(G)$, write $\rho(\sigma)=e^{2\pi
ic}id$ for $c\in[0,1)$, and set $\chi_{\rho}(t)=e^{2\pi ict}$.
Then the pair $(\rho, \chi_{\rho})$ corresponds to a unique
irreducible $\Lambda_{G}(g)$-representation.
\label{lambdabasis}\end{remark}

\begin{example}[$G=\mathbb{Z}/N\mathbb{Z}$]
Let $G=\mathbb{Z}/N\mathbb{Z}$ for $N\geq 1$, and let $\sigma\in
G$. Given an integer $k\in\mathbb{Z}$ which projects to
$\sigma\in\mathbb{Z}/N\mathbb{Z}$, let $x_k$ denote the
representation of $\Lambda_G(\sigma)$ defined by
\begin{equation}\begin{CD}\Lambda_{G}(\sigma)=(\mathbb{Z}\times\mathbb{R})/(\mathbb{Z}(N,0)+\mathbb{Z}(k,1))
@>{[a,t]\mapsto[(kt-a)/N]}>> \mathbb{R}/\mathbb{Z}=\mathbb{T}
@>{q}>> U(1).\end{CD}\label{xk}\end{equation} $R\Lambda_G(\sigma)$
is isomorphic to the ring $\mathbb{Z}[q^{\pm}, x_k]/(x^N_k-q^k)$.
\bigskip

For any finite abelian group
$G=\mathbb{Z}/N_1\mathbb{Z}\times\mathbb{Z}/N_2\mathbb{Z}\times\cdots\times\mathbb{Z}/N_m\mathbb{Z}$,
let $\sigma=(k_1, k_2, \cdots k_n)\in G.$ We have
$$\Lambda_G(\sigma)\cong\Lambda_{\mathbb{Z}/N_1\mathbb{Z}}(k_1)\times_{\mathbb{T}}\cdots\times_{\mathbb{T}}\Lambda_{\mathbb{Z}/N_m\mathbb{Z}}(k_m).$$ Then
\begin{align*}R\Lambda_G(\sigma)&\cong
R\Lambda_{\mathbb{Z}/N_1\mathbb{Z}}(k_1)\otimes_{\mathbb{Z}[q^{\pm}]}\cdots\otimes_{\mathbb{Z}[q^{\pm}]}R\Lambda_{\mathbb{Z}/N_m\mathbb{Z}}(k_m)\\
&\cong\mathbb{Z}[q^{\pm}, x_{k_1}, x_{k_2},\cdots
x_{k_m}]/(x^{N_1}_{k_1}-q^{k_1},x^{N_2}_{k_2}-q^{k_2}, \cdots
x^{N_m}_{k_m}-q^{k_m})\end{align*} where all the $x_{k_j}$'s are
defined as $x_k$ in (\ref{xk}).\label{ppex}
\end{example}

\begin{example}[$G=\mathbb{T}$]
Let $G$ denote the circle group $\mathbb{T}=\mathbb{R}/\mathbb{Z}$. Let $\sigma\in G$ and $c\in\mathbb{R}$ which projects to $\sigma$.
Let $z_c$ denote the representation of $\Lambda_{\mathbb{T}}(\sigma)$ defined by \begin{equation}\begin{CD}\Lambda_{\mathbb{T}}(\sigma)=(\mathbb{R}\times\mathbb{R})/(\mathbb{Z}(1,0)+\mathbb{Z}(c,-1))
@>{[x,t]\mapsto[x+ct]}>> \mathbb{R}/\mathbb{Z}=\mathbb{T}
@>{q}>> U(1).\end{CD}\label{xkt}\end{equation} Observe that $z_{c+1}=qz_c$.
$R\Lambda_{\mathbb{T}}(\sigma)$
is isomorphic to the ring $\mathbb{Z}[q^{\pm}, z_c^{\pm}]$.
\label{ppext}
\end{example}

\begin{example}[$G=\Sigma_3$]\label{replambdasymm3}
$G=\Sigma_3$ has three conjugacy classes represented by $1$,
$(12)$, $(123)$ respectively.
\bigskip

$\Lambda_{\Sigma_3}(1)=\Sigma_3\times\mathbb{T}$, thus,
$R\Lambda_{\Sigma_3}(1)=R\Sigma_3\otimes R\mathbb{T}=\mathbb{Z}[X,
Y]/(XY-Y, X^2-1, Y^2-X-Y-1)\otimes\mathbb{Z}[q^{\pm}]$ where $X$
is the sign representation on $\Sigma_3$ and $Y$ is the standard
representation.
\bigskip

$C_{\Sigma_3}((12))=\langle(12)\rangle=\Sigma_2,$ thus,
$\Lambda_{\Sigma_3}((12))\cong\Lambda_{\Sigma_2}((12)).$ So we
have $R\Lambda_{\Sigma_3}((12))\cong
R\Lambda_{\Sigma_2}((12))=\mathbb{Z}[q^{\pm},
x_1]/(x_1^2-q)\cong\mathbb{Z}[q^{\pm\frac{1}{2}}].$
\bigskip

$C_{\Sigma_3}(123)=\langle(123)\rangle=\mathbb{Z}/3\mathbb{Z},$
thus,
$\Lambda_{\Sigma_3}((123))\cong\Lambda_{\mathbb{Z}/3\mathbb{Z}}(1).$
So we have $R\Lambda_{\Sigma_3}((123))\cong\mathbb{Z}[q^{\pm},
x_1]/(x_1^3-q)\cong\mathbb{Z}[q^{\pm\frac{1}{3}}].$\end{example}

Moreover, we have the conclusion below about the relation between
the induced representations
$Ind|^{\Lambda_G(\sigma)}_{\Lambda_H(\sigma)}(-)$ and
$Ind|^{C_G(\sigma)}_{C_H(\sigma)}(-).$
\begin{lemma}
Let $H$ be a subgroup of $G$ and $\sigma$ an element of $H$. Let
$m$ denote $[C_G(\sigma):C_H(\sigma)]$. Let $V$ denote a
$\Lambda_H(\sigma)-$representation $\lambda\odot_{\mathbb{C}}\chi$
with $\lambda$ a $C_H(\sigma)-$representation, $\chi$ a
$\mathbb{R}-$representation and $\odot_{\mathbb{C}}$ defined in
(\ref{lambdaeq}).

(i)
\begin{equation} res^{\Lambda_G(\sigma)}_{\Lambda_H(\sigma)}(\lambda\odot_{\mathbb{C}}\eta)=(res^{C_G(\sigma)}_{C_H(\sigma)}\lambda)\odot_{\mathbb{C}}\eta.\end{equation}

(ii) The induced representation
$$Ind^{\Lambda_G(\sigma)}_{\Lambda_H(\sigma)}
(\lambda\odot_{\mathbb{C}}\chi)$$ is isomorphic to the
$\Lambda_G(\sigma)-$representation
$$(Ind^{C_G(\sigma)}_{C_H(\sigma)}\lambda)\odot_{\mathbb{C}}\chi.$$
Their underlying vector spaces are both $V^{\oplus m}$.

Thus, the computation of both
$Ind^{\Lambda_G(\sigma)}_{\Lambda_H(\sigma)}
(\lambda\odot_{\mathbb{C}}\chi)$ and
$res^{\Lambda_G(\sigma)}_{\Lambda_H(\sigma)}(\lambda\odot_{\mathbb{C}}\eta)$
can be reduced to the computation of representations of finite
groups. \label{induequ}
\end{lemma}

The proof is straightforward and left to the readers.

\bigskip

Let $k$ be any integer. We describe the relation between
\begin{equation}\Lambda^k_G(g):= C_G(\sigma)\times
\mathbb{R}/\langle (g, -k)\rangle \label{lambdakdef}\end{equation}
and $\Lambda_G(g)$, as well as the relation between their
representation rings.

There is an exact sequence
$$\begin{CD}1 @>>> C_G(g) @>{g\mapsto [g, 0]}>> \Lambda^k_G(g) @>{\pi_k}>>\mathbb{R}/k\mathbb{Z} @>>> 0\end{CD}$$
where the second map $\pi_k: \Lambda^k_G(g)\longrightarrow
\mathbb{R}/k\mathbb{Z}$ is $\pi_k([g, t])= e^{2\pi i t}$.

Let $q^{\frac{1}{k}}: \mathbb{R}/k\mathbb{Z}\longrightarrow U(1)$
denote the composition
$$\begin{CD}\mathbb{R}/k\mathbb{Z} @>{t\mapsto \frac{t}{k}}>> \mathbb{R}/\mathbb{Z} @>{q}>> U(1).\end{CD}$$
The representation ring $R(\mathbb{R}/k\mathbb{Z})$ of
$\mathbb{R}/k\mathbb{Z}$ is  $\mathbb{Z}[q^{\pm\frac{1}{k}}]$. And
there is a canonical isomorphism of $\Lambda-$rings
$$R\Lambda_G(g)\longrightarrow R\Lambda_G^k(g)$$ sending $q$ to
$q^{\frac{1}{k}}$.

Analogous to Lemma \ref{cl}, we have the conclusion about
$R\Lambda^k_G(g)$ below.
\begin{lemma}\label{lambdakrepresentation}
The map $\pi^*_k: R(\mathbb{R}/k\mathbb{Z})\longrightarrow
R\Lambda^k_G(g)$ exhibits it as a free
$\mathbb{Z}[q^{\pm\frac{1}{k}}]-$module. There is a
$\mathbb{Z}[q^{\pm\frac{1}{k}}]-$basis of $R\Lambda^k_G(g)$ given
by irreducible representations $\{\rho_k\}$ such that the
restrictions $\rho_k|_{C_G(g)}$ of them to $C_G(g)$ are precisely
the $\mathbb{Z}$-basis of $RC_G(g)$ given by irreducible
representations.

In other words, any irreducible $\Lambda^k_G(g)-$representation
has the form $\rho\odot_{\mathbb{C}} \chi$ where $\rho$ is an
irreducible representation of $C_G(g)$,
$\chi:\mathbb{R}/k\mathbb{Z}\longrightarrow GL(n, \mathbb{C})$
such that $\chi(k)=\rho(g)$, and
\begin{equation}\rho\odot_{\mathbb{C}} \chi([h, t]):
=\rho(h)\chi(t)\mbox{,     for any   }[h, t]\in
\Lambda^k_G(g).\end{equation}
\end{lemma}

$R\Lambda_G^k(g)$ is a $\mathbb{Z}[q^{\pm}]-$module via the
inclusion $\mathbb{Z}[q^{\pm}]\longrightarrow
\mathbb{Z}[q^{\pm\frac{1}{k}}]$.

\bigskip

There is a group isomorphism $\alpha_k:
\Lambda^k_G(g)\longrightarrow \Lambda_G(g)$ sending $[g, t]$ to
$[g, \frac{t}{k}]$.

Observe that there is a pullback square of groups
\begin{equation}\xymatrix{&\Lambda_G^k(g)\ar[r]^{\alpha_k}\ar[d]^{\pi_k}
&\Lambda_G(g)\ar[d]^{\pi}\\&\mathbb{R}/k\mathbb{Z}
\ar[r]^{t\mapsto
\frac{t}{k}}&\mathbb{R}/\mathbb{Z}}\label{alphaklambdagroup}\end{equation}

By Lemma \ref{lambdakrepresentation}, we can make a
$\mathbb{Z}[q^{\pm\frac{1}{k}}]$-basis
$\{\rho\odot_{\mathbb{C}}\chi_{\rho, k}\}$ for $R\Lambda^k_{G}(g)$
with each $\rho: G\longrightarrow Aut(G)$ an irreducible
$G$-representation and $\chi_{\rho, k}(t)= e^{2\pi i\frac{ct}{k}}$
with $c\in[0,1)$ such that $\rho(\sigma)=e^{2\pi ic}id$. This
collection $\{\rho\odot_{\mathbb{C}}\chi_{\rho, k}\}$ gives  a
$\mathbb{Z}[q^{\pm\frac{1}{k}}]-$basis of $R\Lambda^k_G(g)$.

So we have the commutative square of a pushout square in the
category of $\Lambda-$rings. \begin{equation}\xymatrix{&
R\Lambda^k_G(g)&R\Lambda_G(g)\ar[l]\\
&R(\mathbb{R}/k\mathbb{Z})\ar[u]&R\mathbb{T}\ar[u]\ar[l]}\label{alphaklambdaring}\end{equation}


Moreover we consider
\begin{equation}\Lambda_n(\sigma):=\Lambda_{C_G(\sigma)}(\sigma^n).  \label{lambdakgroup}\end{equation} It is a
subgroup of $\Lambda_{G}(\sigma^n)$. Let $\beta:
\Lambda_n(\sigma)\longrightarrow\Lambda_G(\sigma^n)$ denote the
inclusion. We can define
\begin{equation}\alpha:\Lambda_n(\sigma)\longrightarrow\Lambda_G(\sigma),
(g,t)\mapsto(g, nt).\label{alphalambdan}\end{equation}

We have the pullback square of groups

\begin{equation}
\begin{CD}
\Lambda_n(\sigma)@>\alpha >> \Lambda_G(\sigma)\\
@VVV  @VVV\\
\mathbb{T} @>e^{2\pi i t}\mapsto e^{2\pi int}>> \mathbb{T}\\
\end{CD}
\label{pb}\end{equation}

In addition,  $R\Lambda_n(\sigma)$ is the $\Lambda$-ring pushout
of $\mathbb{Z}[q^{\pm}]\buildrel {R\pi}\over\longrightarrow
R\Lambda_G(\sigma)$ along the inclusion
$\mathbb{Z}[q^{\pm}]\longrightarrow\mathbb{Z}[q^{\pm\frac{1}{n}}].$

\begin{equation}
\begin{CD}
R\Lambda_n(\sigma)@<\alpha^* << R\Lambda_G(\sigma)\\
@AAA  @AAA\\
\mathbb{T} @<R[n]<< \mathbb{T}\\
\end{CD}
\label{po}\end{equation}

In particular, there is a canonical isomorphism of $\Lambda$-rings
\begin{equation}R\Lambda_n(\sigma)\buildrel\sim\over\longrightarrow
R\Lambda_G(\sigma)[
q^{\pm\frac{1}{n}}].\label{lambdaisoalpha}\end{equation}

\bigskip

\subsection{Quasi-elliptic cohomology}\label{orbqec} In this
section we introduce the definition of quasi-elliptic cohomology
$QEll^*_G$ in terms of orbifold K-theory, and then express it via
equivariant K-theory. We assume familiarity with
\cite{SegalequiK}. The reader may read Chapter 3 in \cite{ALRuan}
and \cite{Moe02} for a reference of orbifold K-theory.

Quasi-elliptic cohomology is defined from the  full subgroupoid
$\Lambda(X/\!\!/G)$ of the orbifold loop space $L_{orb}(X/\!\!/G)$
consisting of constant loops. Before describing $\Lambda(X/\!\!/G)$ in detail, we
recall what Inertia groupoid is. A reference for that is
\cite{LU01}.


\begin{definition}Let $\mathbb{G}$ be a groupoid. The Inertia
groupoid $I(\mathbb{G})$ of $\mathbb{G}$ is defined as follows.

An object $a$ is an arrow in $\mathbb{G}$ such that its source and
target are equal. 
A morphism $v$ joining two objects $a$ and $b$ is an arrow $v$ in
$\mathbb{G}$ such that $$v\circ a=b\circ v.$$ In other words, $b$
is the conjugate of $a$ by $v$, $b=v\circ a\circ v^{-1}$.

The torsion Inertia groupoid  $I^{tors}(\mathbb{G})$ of
$\mathbb{G}$ is a full subgroupoid
of of $I(\mathbb{G})$ with only objects of finite order.  
\label{inertiagroupoid}
\end{definition}

Let $G$ be a compact Lie group and $X$ a $G-$space.
\begin{example}The
torsion inertia groupoid $I^{tors}(X/\!\!/G)$ of the translation
groupoid $X/\!\!/G$ is the groupoid with

\textbf{objects}: the space $\coprod\limits_{g\in G^{tors}}X^{g}$

\textbf{morphisms}: the space $\coprod\limits_{g, g'\in
G^{tors}}C_G(g,g')\times X^g$ where $C_G(g,g')=\{\sigma\in
G|g'\sigma=\sigma g\}\subseteq G.$

For $x\in X^g$ and $(\sigma, g)\in C_G(g,g')\times X^g$, $(\sigma,
g)(x)=\sigma x\in X^{g'}.$ \label{torsionquotient}
\end{example}


\begin{definition} The groupoid $\Lambda(X/\!\!/G)$ has the same objects as
$I^{tors}(X/\!\!/G)$ but richer morphisms
$$\coprod\limits_{g, g'\in G^{tors}}\Lambda_G(g, g')\times X^g$$
where $\Lambda_G(g, g')$ is defined in Definition \ref{lambdak1}. For an object $x\in X^g$ and a morphism
$([\sigma, t], g)\in \Lambda_G(g, g')\times X^g$, $([\sigma, t],
g)(x)=\sigma x\in X^{g'}.$ The composition of the morphisms is
defined by
\begin{equation}[\sigma_1, t_1][\sigma_2,
t_2]=[\sigma_1\sigma_2, t_1+t_2].\end{equation} We have a
homomorphism of orbifolds
$$\pi: \Lambda(X/\!\!/G)\longrightarrow\mathbb{T}$$ sending all the objects to the single object in $\mathbb{T}$, and
a morphism $([\sigma,t], g)$ to $e^{2\pi it}$ in $\mathbb{T}$.

\end{definition}

\begin{definition} The quasi-elliptic cohomology $QEll^*_G(X)$ is defined to be
$K^*_{orb}(\Lambda(X/\!\!/G))$. \end{definition}We can unravel the definition and
express it via equivariant K-theory.

Let $\sigma\in G^{tors}$. The fixed point
space $X^{\sigma}$ is a $C_G(\sigma)-$space. We can define a
$\Lambda_G(\sigma)-$action on $X^{\sigma}$ by
$$[g, t]\cdot x:=g\cdot x.$$
Then we have

\begin{proposition}
\begin{equation}QEll^*_G(X)=\prod_{g\in
G^{tors}_{conj}}K^*_{\Lambda_G(g)}(X^{g})=\bigg(\prod_{g\in
G^{tors}}K^*_{\Lambda_G(g)}(X^{g})\bigg)^G.\end{equation}
\end{proposition}


Thus, for each $g\in\Lambda_G(g)$, we can define the projection
$$\pi_g: QEll^*_G(X)\longrightarrow K^*_{\Lambda_G(g)}(X^{g}).$$
For the single point space, we have
\begin{equation}QEll^0_G(\mbox{pt})\cong\prod_{g\in G^{tors}_{conj}} R\Lambda_G(g). \label{qecpt}\end{equation}
\begin{remark}
According to Theorem 8.7.5, \cite{KM85} there is  a
smooth one-dimensional commutative group scheme $T$ over
$\mathbb{Z}[q^{\pm}]$ such that we have the unique isomorphism of ind-group-schemes on
$\mathbb{Z}((q))$ $$T_{torsion}\otimes_{\mathbb{Z}[q^{\pm}]}
\mathbb{Z}((q))\buildrel\sim\over\longrightarrow Tate(q)_{tors}$$ where $Tate(q)_{tors}$ is the torsion part of the Tate curve and $T_{torsion}$
is the torsion part of $T$.

The group scheme $T$ is discussed in Section 8.7, \cite{KM85}. The $N-$torsion points
$T[N]$ of it is the disjoint union of $N$ schemes $T_0[N]$,
$\cdots$ $T_{N-1}[N]$, where
$$T_i[N]=\mbox{Spec}(\mathbb{Z}[q^{\pm}][x]/(x^N-q^i)).$$

By the computation in Example \ref{ppex},  we have $$QEll_{\mathbb{Z}/N\mathbb{Z}}(\mbox{pt})=\prod_{m=0}^{N-1}K_{\Lambda_{\mathbb{Z}/N\mathbb{Z}}(m)}(\mbox{pt})
=\prod_{m=0}^{N-1}R\Lambda_{\mathbb{Z}/N\mathbb{Z}}(m)=\prod_{m=0}^{N-1}\mathbb{Z}[q^{\pm}, x_m]/(x^N_m-q^m).$$

Thus, we have the relation \begin{equation} \mbox{Spec}(QEll_{\mathbb{Z}/N\mathbb{Z}}(\mbox{pt}))\cong T[N].\end{equation}

Analogously, by the computation in Example \ref{ppext}, we have \begin{equation} \mbox{Spec}(QEll_{\mathbb{T}}(\mbox{pt}))\cong T.\end{equation} \label{KMgroup}
\end{remark}

By computing the representation rings of $R\Lambda_G(g)$, we get
$QEll^*_G(-)$ for contractible spaces. Then, using Mayer-Vietoris
sequence, we can compute $QEll^*_G(-)$ for any $G$-CW complex by
patching the $G$-cells together.

We have the ring homomorphism
$$\mathbb{Z}[q^{\pm}]=K^0_{\mathbb{T}}(\mbox{pt})\buildrel{\pi^*}\over\longrightarrow K^0_{\Lambda_G(g)}(\mbox{pt})\longrightarrow
K^0_{\Lambda_G(g)}(X)$$ where $\pi: \Lambda_G(g)\longrightarrow
\mathbb{T}$ is the projection defined in (\ref{pizq}) and the
second is via the collapsing map $X\longrightarrow \mbox{pt}$. So
$QEll_G^*(X)$ is naturally a
$\mathbb{Z}[q^{\pm}]-$algebra. 


The $\Lambda$-ring structure on $QEll^*_G(X)$ is the direct
product of the exterior algebra on each equivariant $K$-group,
with componentwise multiplication. 

For each $QEll^*_G$ we can equip a special family of
$\Lambda-$ring homomorphisms
$$\mu^n: QEll^*_G(X)\cong\prod_{g\in
G^{tors}_{conj}}K^*_{\Lambda_G(g)}(X^{g})\longrightarrow\prod_{g\in
G^{tors}_{conj}}K^*_{\Lambda_G^n(g)}(X^{g})\cong
QEll^*_G(X)[q^{\pm\frac{1}{n}}]$$ defined by
\begin{equation}QEll^*_G(X)\longrightarrow
K^*_{\Lambda_G(g^n)}(X^{g^n})\buildrel\beta^*\over\longrightarrow
K^*_{\Lambda_G^n(g)}(X^{g^n})\longrightarrow
K^*_{\Lambda_G^n(g)}(X^{g}),\label{munconstr}\end{equation} where
the first map is projection, the second is the restriction via the
inclusion $\beta: \Lambda_G^n(g)\longrightarrow\Lambda_G(g^n)$, and
the third is restriction along $X^{g}\subseteq X^{g^n}.$

\bigskip

In addition, we can express the $\Lambda-$ring isomorphism
(\ref{lambdaisoalpha}) and this family of $\Lambda-$ring
homomorphisms $\{\mu^n\}_n$ in terms of orbifold K-theory, which
are fairly neat.

Let $\Lambda_n(g,g')$ denote the quotient of $C_G(g,
g')\times\mathbb{R}$ under the action of $\mathbb{Z}$ where the
action of the generator of $\mathbb{Z}$ is given by $(\sigma,
t)\mapsto (\sigma g^n, t-1)=(g^n\sigma', t-1).$ Then we can define
another groupoid $\Lambda_n(X/\!\!/G)$ with the same objects as
$\Lambda(X/\!\!/G)$ and morphisms $$\coprod\limits_{g, g'\in
G^{tors}}\Lambda_n(g,g')\times X^g$$ such that for each $x\in
X^g$, $([\sigma, t], g)(x)=\sigma x\in X^{g'}$. We can also define
$\pi: \Lambda_n(X/\!\!/G)\longrightarrow\mathbb{T}$ sending all the
objects to the single object in $\mathbb{T}$ and a morphism
$([\sigma, t], g)$ to the morphism $e^{2\pi i t}$ in $\mathbb{T}$.

Let $\alpha: \Lambda_n(X/\!\!/G)\longrightarrow\Lambda(X/\!\!/G)$ be the 
homomorphism of orbifolds sending an object $x\in X^g$ to $x\in
X^{g}$ and a morphism $[\sigma, t]: x\longrightarrow x'$ to
$[\sigma, nt]: x\longrightarrow x'$.   
Let $\beta:\Lambda_n(X/\!\!/G)\longrightarrow\Lambda(X/\!\!/G)$ be the
functor sending an object $x\in X^g$ to $x\in X^{g^n}$ and a
morphism $[\sigma, t]: x\longrightarrow x'$ to $[\sigma, t]:
x\longrightarrow x'$.  

Since we have the pullback square of groups (\ref{alphaklambdagroup}) and the
pushout square of groups (\ref{alphaklambdaring}), we have the pullback square 
of groupoids
\begin{equation}
\begin{CD}
\Lambda_n(X/\!\!/G)@>\alpha >> \Lambda(X/\!\!/G)\\
@VVV  @VVV\\
\mathbb{T} @>{e^{2\pi i t}\mapsto e^{2\pi n i t}}>> \mathbb{T},\\
\end{CD}
\label{pb3}\end{equation} 
and the pushout square in the category of $\Lambda$-rings

\begin{equation}
\begin{CD}
K^*_{orb}(\Lambda_n(X/\!\!/G))@<\alpha^* << K^*_{orb}(\Lambda(X/\!\!/G))\\
@AAA  @AAA\\
K^*_{orb}(*/\!\!/\mathbb{T}) @<<< K^*_{orb}(*/\!\!/\mathbb{T}).\\
\end{CD}
\label{po3}\end{equation}

It induces a natural isomorphism
\begin{equation}K^*_{orb}(\Lambda(X/\!\!/G))[q^{\pm\frac{1}{n}}]\buildrel\sim\over\longrightarrow
K^*_{orb}(\Lambda_n(X/\!\!/G)).\label{lambdaisoorb}\end{equation}

The $\Lambda$-ring homorphism $\mu^n$ can be constructed by
$$\mu^n: K^*_{orb}(\Lambda(X/\!\!/G))\buildrel
\beta^*\over\longrightarrow K^*_{orb}(\Lambda_n(X/\!\!/G))\cong
K_{orb}^*(\Lambda(X/\!\!/G))[q^{\pm\frac{1}{n}}].$$

\subsection{Properties}\label{propertiesqec} $QEll^*_G$ inherits
most properties of equivariant $K$-theory. In this section
we discuss some properties of $QEll^*_G$,
including the restriction map, the K\"{u}nneth map on it, its
tensor product and the change of group isomorphism.

Since each homomorphism $\phi: G\longrightarrow H$ induces a
well-defined homomorphism $\phi_{\Lambda}:
\Lambda_G(\tau)\longrightarrow\Lambda_H(\phi(\tau))$ for each
$\tau$ in $G^{tors}$, we can get the proposition below directly.
\begin{proposition}For each homomorphism $\phi: G\longrightarrow H$, it induces a ring map
$$\phi^*: QEll^*_H(X)\longrightarrow QEll^*_G(\phi^*X)$$ characterized by the commutative diagrams

\begin{equation}\begin{CD}QEll^*_H(X) @>{\phi^*}>> QEll^*_G(\phi^*X) \\ @V{\pi_{\phi(\tau)}}VV  @V{\pi_{\tau}}VV  \\
K^*_{\Lambda_H(\phi(\tau))}(X^{\phi(\tau)}) @>{\phi^*_{\Lambda}}>>
K^*_{\Lambda_G(\tau)}(X^{\phi(\tau)})\end{CD}\end{equation} for any $\tau \in G^{tors}$. So $QEll^*_G$ is functorial in $G$.
\label{restrictionq}\end{proposition}

More generally, we have the restriction map below.
\begin{proposition}
For any groupoid homomorphism $\phi: X/\!\!/G\longrightarrow Y/\!\!/H$, we
have the groupoid homomorphism $\Lambda(\phi):
\Lambda(X/\!\!/G)\longrightarrow \Lambda(Y/\!\!/H)$ sending an object $(x,
g)$ to $(\phi(x), \phi(g))$, and a morphism $([\sigma, t], g)$ to
$([\phi(\sigma), t], \phi(g))$. Thus, we get a ring map
$$\phi^*: QEll^*(Y/\!\!/H)\longrightarrow QEll^*(X/\!\!/G)$$ characterized by the commutative diagrams

\begin{equation}\begin{CD}QEll^*(Y/\!\!/H) @>{\phi^*}>> QEll^*(X/\!\!/G) \\ @V{\pi_{\phi(\tau)}}VV  @V{\pi_{\tau}}VV  \\
K^*_{\Lambda_H(\phi(\tau))}(Y^{\phi(\tau)}) @>{\phi^*_{\Lambda}}>>
K^*_{\Lambda_G(\tau)}(X^{\tau})\end{CD}\end{equation}  for any $\tau \in G^{tors}$.
\label{functorialbothqell}
\end{proposition}

Moreover, we can define K\"{u}nneth map on quasi-elliptic cohomology
induced from that on equivariant $K$-theory.

Let $G$ and $H$ be two compact Lie groups. $X$ is a $G$-space and
$Y$ is a $H$-space. Let $\sigma\in G^{tors}$ and $\tau\in
H^{tors}$. Let
$\Lambda_G(\sigma)\times_{\mathbb{T}}\Lambda_H(\tau)$ denote the
fibered product of the morphisms
$$\Lambda_G(\sigma)\buildrel{\pi}\over\longrightarrow
\mathbb{T}\buildrel{\pi}\over\longleftarrow\Lambda_H(\tau).$$ It is
isomorphic to $\Lambda_{G\times H}(\sigma, \tau)$ under the
correspondence
$$([\alpha, t], [\beta, t])\mapsto [\alpha, \beta, t].$$

Consider the map below
\begin{align*}T: K_{\Lambda_G(\sigma)}(X^{\sigma})\otimes
K_{\Lambda_H(\tau)}(Y^{\tau})&\longrightarrow
K_{\Lambda_{G}(\sigma)\times\Lambda_{H}(\tau)}(X^{\sigma}\times
Y^{\tau})\buildrel{res}\over\longrightarrow
K_{\Lambda_{G}(\sigma)\times_{\mathbb{T}}\Lambda_{H}(\tau)}(X^{\sigma}\times
Y^{\tau})\\ &\buildrel{\cong}\over\longrightarrow
K_{\Lambda_{G\times H}(\sigma, \tau)}((X\times Y)^{(\sigma,
\tau)}).\end{align*} where the first map is the K\"{u}nneth map of
equivariant K-theory, the second is the restriction map  and the
third is the isomorphism induced by the group isomorphism
$\Lambda_{G\times H}(\sigma,
\tau)\cong\Lambda_G(\sigma)\times_{\mathbb{T}}\Lambda_H(\tau)$.

For $g\in G^{tors}$, let $1$ denote
the trivial line bundle over $X^g$ and let $q$ denote the line
bundle $1\odot q$ over $X^g$. The map $T$ above sends both
$1\otimes q$ and $q\otimes 1$  to $q$. So we get the well-defined
map
\begin{equation}K^*_{\Lambda_G(\sigma)}(X^{\sigma})\otimes_{\mathbb{Z}[q^{\pm}]}K^*_{\Lambda_H(\tau)}(Y^{\tau})\longrightarrow
K_{\Lambda_{G\times H}(\sigma, \tau)}((X\times
Y)^{(\sigma, \tau)}).\label{ku}\end{equation}

\begin{definition}The tensor produce of quasi-elliptic cohomology is defined by
\begin{equation}QEll^*_G(X)\widehat{\otimes}_{\mathbb{Z}[q^{\pm}]}QEll^*_H(Y)
\cong\prod_{\sigma\in G^{tors}_{conj}\mbox{,   } \tau\in
H^{tors}_{conj}}K^*_{\Lambda_G(\sigma)}(X^{\sigma})\otimes_{\mathbb{Z}[q^{\pm}]}K^*_{\Lambda_H(\tau)}(Y^{\tau}).\label{qectensor}\end{equation}
The direct product of the maps defined in (\ref{ku}) gives a ring
homomorphism
$$QEll^*_G(X)\widehat{\otimes}_{\mathbb{Z}[q^{\pm}]}QEll^*_H(Y)\longrightarrow
QEll^*_{G\times H}(X\times Y),$$ which is the K\"{u}nneth map of
quasi-elliptic cohomology.
\end{definition}

By Lemma \ref{cl} we have
$$QEll^*_G(\mbox{pt})\widehat{\otimes}_{\mathbb{Z}[q^{\pm}]}QEll^*_H(\mbox{pt})=QEll^*_{G\times H}(\mbox{pt}).$$
More generally, we have the proposition below.

\begin{proposition}
Let $X$ be a $G\times
H-$space with trivial $H-$action and let $\mbox{pt}$ be the single
point space with trivial $H-$action.
Then we have
$$QEll_{G\times H}(X)\cong QEll_G(X)\widehat{\otimes}_{\mathbb{Z}[q^{\pm}]} QEll_H(pt).$$

Especially, if $G$ acts trivially on $X$, we have
$$QEll_G(X)\cong   QEll(X)\widehat{\otimes}_{\mathbb{Z}[q^{\pm}]}
QEll_G(\mbox{pt}).$$ Here $QEll^*(X)$ is
$QEll^*_{\{e\}}(X)=K^*_{\mathbb{T}}(X)$.
\end{proposition}

\begin{proof}

\begin{align*}QEll_{G\times H}(X) &=\prod\limits_{\substack{g\in
G^{tors}_{conj}
\\ h\in H^{tors}_{conj}}}K_{\Lambda_{G\times H}(g, h)}(X^{(g,
h)})\cong \prod\limits_{\substack{g\in G^{tors}_{conj} \\ h\in
H^{tors}_{conj}} }K_{\Lambda_{G}(g)\times_{\mathbb{T}}
\Lambda_{H}(h)}(X^{g})\\ &\cong \prod\limits_{\substack{g\in
G^{tors}_{conj} \\ h\in H^{tors}_{conj}}}
K_{\Lambda_{G}(g)}(X^{g})\otimes_{\mathbb{Z}[q^{\pm}]}
K_{\Lambda_H(h)}(\mbox{pt})=
QEll_G(X)\widehat{\otimes}_{\mathbb{Z}[q^{\pm}]} QEll_H(pt).
\end{align*}

\end{proof}


\begin{proposition}
If $G$ acts freely on $X$,
$$QEll^*_G(X)\cong QEll^*_e(X/G).$$
\end{proposition}
\begin{proof}
Since $G$ acts freely on $X$, $$X^{\sigma}=\begin{cases}\emptyset,
&\text{if $\sigma\neq e$;}\\ X, &\text{if
$\sigma=e$.}\end{cases}$$ Thus,
$QEll^*_G(X)\cong\prod\limits_{\sigma\in
G^{tors}_{conj}}K^*_{\Lambda_G(\sigma)/C_G(\sigma)}(X^{\sigma}/C_G(\sigma))\cong
K^*_{\mathbb{T}}(X/G).$

Since $\mathbb{T}$ acts trivially on $X$, we have
$K^*_{\mathbb{T}}(X/G)=QEll^*_e(X/G)$ by definition. And it is
isomorphic to $K^*(X/G)\otimes R\mathbb{T}$.\end{proof}


We also have the change-of-group isomorphism as  in
equivariant $K$-theory.

Let $H$ be a closed subgroup of $G$ and $X$ a $H$-space. Let
$\phi: H\longrightarrow G$ denote the inclusion homomorphism. The
change-of-group map $\rho^G_H: QEll^*_G(G\times_HX)\longrightarrow
QEll^*_H(X)$ is defined as the composite
\begin{equation}QEll^*_G(G\times_HX)\buildrel{\phi^*}\over\longrightarrow
QEll^*_H(G\times_H X)\buildrel{i^*}\over\longrightarrow
QEll_H^*(X)\label{changeofgroup}
\end{equation}
where $\phi^*$ is the restriction map and $i: X\longrightarrow
G\times_HX$ is the $H-$equivariant map defined by $i(x)=[e, x].$

\begin{proposition} The change-of-group map
$$\rho^G_H: QEll^*_G(G\times_H X)\longrightarrow
QEll^*_H(X)$$ defined in (\ref{changeofgroup}) is an
isomorphism.\end{proposition}
\begin{proof}
For any $\tau\in H^{tors}_{conj}$, there exists a unique
$\sigma_{\tau}\in G_{conj}^{tors}$ such that
$\tau=g_{\tau}\sigma_{\tau}g_{\tau}^{-1}$ for some $g_{\tau}\in
G$.  Consider the maps \begin{equation}\begin{CD}
\Lambda_G(\tau)\times_{\Lambda_H(\tau)}X^{\tau}@>{[[a, t], x
]\mapsto [a, x]}>> (G\times_H X)^{\tau}@>{[u, x]\mapsto
[g_{\tau}^{-1}u, x]}>>
(G\times_HX)^{\sigma}.\end{CD}\end{equation} The first map is
$\Lambda_G(\tau)-$equivariant and the second is equivariant with
respect to the homomorphism $c_{g_{\tau}}:
\Lambda_{G}(\sigma)\longrightarrow \Lambda_G(\tau)$ sending $[u,
t]\mapsto [g_{\tau} u g_{\tau}^{-1}, t]$. Taking a coproduct over
all the elements $\tau\in H^{tors}_{conj}$ that are conjugate to
$\sigma\in G^{tors}_{conj}$ in $G$, we get an isomorphism
$$\gamma_{\sigma}: \coprod_{\tau}\Lambda_G(\tau)\times_{\Lambda_H(\tau)} X^{\tau}\longrightarrow
(G\times_HX)^{\sigma}$$ which is $\Lambda_G(\sigma)-$equivariant
with respect to $c_{g_{\tau}}$. Then we have the map
\begin{equation}\gamma:=\prod_{\sigma\in G^{tors}_{conj}}\gamma_{\sigma}:
\prod_{\sigma\in G^{tors}_{conj}}
K^*_{\Lambda_G(\sigma)}(G\times_HX)^{\sigma}\longrightarrow
\prod_{\sigma\in
G^{tors}_{conj}}K^*_{\Lambda_G(\sigma)}(\coprod_{\tau}\Lambda_G(\tau)\times_{\Lambda_H(\tau)}
X^{\tau})
\end{equation}

It's straightforward to check the change-of-group map coincide
with the composite \begin{align*} QEll^*_{G}(G\times_H
X)\buildrel{\gamma}\over\longrightarrow \prod_{\sigma\in
G^{tors}_{conj}}K^*_{\Lambda_G(\sigma)}(\coprod_{\tau}\Lambda_G(\tau)\times_{\Lambda_H(\tau)}
X^{\tau})\longrightarrow &\prod_{\tau\in
H^{tors}_{conj}}K^*_{\Lambda_H(\tau)}(X^{\tau})\\&=QEll^*_{H}(X)\end{align*}
with  the second map  the change-of-group isomorphism in
equivariant $K-$theory.
\end{proof}

\subsection{Formulas for Induction} \label{inductionqell} 


In this section \ref{inductionqell} we introduce the induction
formula for quasi-elliptic cohomology. The induction formula for Tate
K-theory is constructed in Section 2.3.3, \cite{Gan13}.

Let $H\subseteq G$ be an inclusion of compact Lie groups and $X$
be a $G-$space. Then we have the inclusion of the groupoids
$$j:X/\!\!/H\longrightarrow X/\!\!/G.$$

Let $a'=\prod\limits_{\sigma\in H^{tors}_{conj}}a'_{\sigma}$ be an
element in $QEll_H(X)=\prod\limits_{\sigma\in
H^{tors}_{conj}}K_{\Lambda_H(\sigma)}(X^{\sigma})$ where $\sigma$ goes
over all the conjugacy classes in $H$. The finite covering map
$$f': \Lambda(G\times_H X/\!\!/G)\longrightarrow  \Lambda(X/\!\!/G)$$  is defined by sending  an object $(\sigma, [g, x])$ to  $(\sigma, gx)$ and
a morphism  $([g', t], (\sigma, [g, x]))$ to $([g', t], (gx,
\sigma))$. The transfer of  quasi-elliptic cohomology
$$\mathcal{I}_H^G: QEll_H(X)\longrightarrow QEll_G(X)$$
is defined to be the composition \begin{equation}
QEll_H(X)\buildrel\cong\over\longrightarrow
QEll_G(G\times_HX)\longrightarrow
QEll_G(X)\label{qelltransfer}\end{equation} where the first map is
the change-of-group isomorphism and the second is the finite
covering.

Thus
$$\mathcal{I}^G_H(a')_{g}=\sum_{r}r\cdot a'_{r^{-1}gr}$$ where $r$ goes
over a set of representatives of $(G/H)^{g}$, in other words,
$r^{-1}gr$ goes over a set of representatives of conjugacy classes
in $H$ conjugate to $g$ in $G$.

\begin{equation}\mathcal{I}^G_H(a')_{g}=\begin{cases}Ind^{\Lambda_G}_{\Lambda_H}(a'_{g}) &\mbox{if  }g\mbox{ is conjuate to some element  }h \mbox{  in
}H;\\ 0 &\mbox{if there is no element conjugate to }g\mbox{  in
}H.
\end{cases}\end{equation}

\bigskip
There is another way to describe the transfer, which is shown in Rezk's unpublished work
\cite{Rez11} for quasi-elliptic cohomology. 
The transfer of Tate K-theory can be described similarly.

\section{Orbifold quasi-elliptic cohomology}\label{orbifoldquasibeforepower}

The elliptic cohomology of orbifolds involves a rich interaction
between the orbifold structure and the elliptic curve. Orbifold quasi-elliptic cohomology can also be constructed from loop spaces via bibundles.
We give the construction via bibundles in Section \ref{busb}.
Ganter explores this interaction in \cite{Gan13} in the case of the Tate
curve, describing $K_{Tate}$ for an orbifold $X$ in term of the
equivariant K-theory and the groupoid structure of $X$. We show the relation between orbifold quasi-elliptic cohomology
and orbifold Tate K-theory in
Section \ref{sb1}.

\subsection{Definition}\label{busb}

In this section we construct orbifold quasi-elliptic cohomology via loop space. The idea is similar to that in Section \ref{loopmodel}.
For the definition of groupoid action and groupoid-principal bundles, the readers can refer to Section 3, \cite{LerStack}.

Let $X$ be an orbifold groupoid.

\begin{definition}[$Loop_1(X)$]

We use
$Loop_1(X)$ to denote the category $Bibun(S^1/\!\!/\ast,
X)$, which generalizes Definition \ref{loopspacemorphism}.
According to Definition \ref{bibundle}, each object
consists of a smooth manifold $P$ and two structure maps
$P\buildrel{\pi}\over\longrightarrow S^1$ a smooth principal
$X-$bundle  and $f: P\longrightarrow X_0$ an $X-$equivariant map.  A morphism 
is an $X-$bundle map $\alpha: P\longrightarrow P'$ making the
diagram below commute.
$$\xymatrix{S^1
&P \ar[l]_{\pi}\ar[d]^{\alpha}\ar[r]^{f} & X_0 \\
&P' \ar[lu]^{\pi'}\ar[ru]_{f'} &}$$ Thus, the morphisms in
$Loop_1(X)$ from $P$ to $P'$ are $X-$isomorphisms.

\label{loopspacemorB}\end{definition}

Next we add rotations to the groupoid $Loop_1(X)$ and give the definition of the groupoid $Loop^{ext}_1(X)$ which generalizes Definition \ref{loopext3space}.

\begin{definition}[$Loop^{ext}_1(X)$]\label{loopextB} 
Let $Loop^{ext}_1(X)$ denote the groupoid with the same
objects as $Loop_1(X)$. Each morphism
consists of the pair $(t, \alpha)$ where $t\in\mathbb{T}$ is a
rotation and $\alpha$ is an $X-$bundle map. They make the diagram
below commute.
$$\xymatrix{S^1\ar[d]_{t}
&P \ar[l]_{\pi}\ar[d]_{\alpha}\ar[r]^{f} & X_0 \\S^1 &P'
\ar[l]^{\pi'}\ar[ru]_{f'} &}$$

\end{definition}

In addition, we can define the groupoid of ghost loops for orbifolds.
\begin{definition}[Ghost Loops]The ghost loops corresponds to the full subgroupoid
$GhLoop(X)$ of $Loop^{ext}_1(X)$ consisting of objects
$S^1\leftarrow P\buildrel{\widetilde{\delta}}\over\rightarrow X_0$
such that $\widetilde{\delta}(P)\subseteq X_0$ contained in a single
$G-$orbit. \label{ghostloopB} \end{definition}

The groupoid constant loops $\Lambda(X)$ is a subgroupoid of $GhLoop(X)$.
\begin{definition} The groupoid $\Lambda(X)$ is the subgroupoid of $Loop^{ext}_1(X)$ consisting of objects $S^1\leftarrow P\buildrel{\widetilde{\delta}}\over\rightarrow X_0$
such that there exists a section of $s_P: P\longrightarrow S^1$ such that $f\circ s_P$ is a constant map. Let $\{x_P\}$ denote the image of $f\circ s_P$. Each object is
determined by $x_P$ and an automorphism $g\in aut(x_P)$ in $X$ of finite order.

In each morphism $$\xymatrix{S^1\ar[d]_{t}
&P \ar[l]_{\pi}\ar[d]_{\alpha}\ar[r]^{f} & X_0 \\S^1 &P'
\ar[l]^{\pi'}\ar[ru]_{f'} &}$$
$\alpha \in Mor_X(x_P, x_{P'})$ and the morphism $(t+1, \alpha)$ is the same as $(t, \alpha\circ g)$.

When $X$ is a global quotient $M/\!\!/G$, $\Lambda(X)$ is isomorphic to the groupoid $\Lambda(M/\!\!/ G).$
\label{lambdaorb}
\end{definition}

\begin{definition}The orbifold quasi-elliptic cohomology of $X$ is defined to be \begin{equation}QEll^*(X):=K_{orb}^*(\Lambda(X)). \label{orbqelldef}\end{equation} \end{definition}
In the global quotient case,
$$QEll^*(M/\!\!/G)=QEll^*_G(M).$$

\subsection{Relation with Orbifold Tate K-theory}\label{sb1}

In this section we give another definition of orbifold quasi-elliptic cohomology equivalent to Definition \ref{orbqelldef}. It is closely related to Ganter's construction
of orbifold Tate K-theory in \cite{Gan13}.

First we recall some relevant
constructions and notations. The main reference is \cite{Gan13}.

Consider the category of groupoids $\mathcal{G}pd$ as a 2-category
with small topological groupoids as the objects and with
$$\mbox{1Hom}(X, Y)=Fun(X, Y),$$ 
the groupoid of continuous functors from $X$  to $Y$.

\begin{definition}The center of a groupoid $X$ is defined to be the group
$$\mbox{Center}(X):=\mbox{2Hom}(\mbox{Id}_X, \mbox{Id}_X)=\mathcal{N}at(\mbox{Id}_X, \mbox{Id}_X)$$ of natural
transformations from Id$_x$ to Id$_x$. \end{definition}

\begin{definition}Let $\mathcal{G}pd^{cen}$ denote the 2-category whose objects are
pairs $(X,\xi)$ with $\xi$ a center element of $X$, and the set of
morphisms from $(X, \xi)$ to $(Y, \nu)$ is $$\mbox{1Hom}((X, \xi),
(Y, \nu))\subset Fun(X, Y)$$ with $$f\xi=\nu f$$ for each morphism
$f$.
\end{definition}

We will assume all the center elements have finite order.

\begin{example}
If $G$ is a finite group, $\mbox{Center}(pt/\!\!/G)$ is the center of
the group $G$.

\end{example}

\begin{example}The Inertia groupoid $I(X)$  of a groupoid $X$, which is
defined in Definition \ref{inertiagroupoid}, is isomorphic to
$$Fun(\mbox{pt}/\!\!/\mathbb{Z}, X).$$ Each object of
$I(X)$ can be viewed as pairs $(x, g)$ with $x\in \mbox{ob}(X)$
and $g\in aut(x)$, $gx=x$. A morphism from $(x_1, g_1)$ to $(x_2,
g_2)$ is a morphism $h: x_1\longrightarrow x_2$ in $X$ satisfying
$h\circ g_1=g_2\circ h$ in $X$. So in $I(X)$,
$$\mbox{Hom}((x_1, g_1), (x_2, g_2))=\{h: x_1\longrightarrow x_2| h\circ g_1=g_2\circ
h\}.$$

Recall $I^{tors}(X)$ is  a full subgoupoid of $I(X)$ with elements
$(x, g)$ where $g$ is of finite order. Let $\xi^k$ denote the
center element of $I^{tors}(X)$ sending $(x, g)$ to $(x, g^k)$. We use $\xi$ to denote $\xi^1$.

For any $k\in \mathbb{Z}$, we have the 2-functor
\begin{align*}\mathcal{G}pd &\longrightarrow \mathcal{G}pd^{cen}\\
X&\mapsto (I^{tors}(X), \xi^k).
\end{align*}

\end{example}

\begin{example}In the global quotient case,   as indicated in Example
\ref{torsionquotient}, $I^{tors}(X/\!\!/G)$ is isomorphic to
$\prod\limits_{g\in G^{tors}_{conj}}X^g/\!\!/C_G(g)$. The center
element $\xi^k|_{X^g}=g^k$.

\end{example}

\begin{definition}Let $\mbox{pt}/\!\!/\mathbb{R}\times_{1\sim\xi}I^{tors}(X)$ denote the
groupoid
$$(\mbox{pt}/\!\!/\mathbb{R})\times I^{tors}(X)/\sim$$ with
$\sim$ generated by $1\sim\xi$.\end{definition}

\begin{lemma}The groupoid $\mbox{pt}/\!\!/\mathbb{R}\times_{1\sim\xi}I^{tors}(X)$  is isomorphic to $\Lambda(X)$.

Thus, $QEll^*(X)$ is isomorphic to
\begin{equation}K^*_{orb}(\mbox{pt}/\!\!/\mathbb{R}\times_{1\sim
\xi}I^{tors}(X)).\label{orbqecnora}\end{equation}\label{equivorbqell}\end{lemma}
The proof of Lemma \ref{equivorbqell} is left to the readers.



\begin{remark}
Orbifold quasi-elliptic cohomology $QEll(X)$ can be defined to be a
subring of  $K_{orb}(X)\llbracket q^{\pm\frac{1}{|\xi|}}
\rrbracket$ that is the Grothendieck group of finite sums
$$\sum_{a\in\mathbb{Q}} V_a q^a$$ satisfying: 
$$\mbox{for each }a\in\mathbb{Q} \mbox{, the coefficient }V_a \mbox{ is an }e^{2\pi ia}-\mbox{eigenbundle of }\xi.$$
\end{remark}




\begin{thebibliography}{a}


\bibitem{Isogeny} Matthew Ando:  \textsl{Isogenies of formal group laws and power operations
in
  the cohomology theories $E_n$}, Duke Math. J. 79(1995), no. 2, 423$\textendash$485.

\bibitem{Powerma}Matthew Ando: \textsl{Power operations in elliptic cohomology and representations of loop groups}, Trans. Amer. Math. Soc. 352(2000), no. 12, 5619$\textendash$5666.

\bibitem{AHS}Matthew Ando, Michael J. Hopkins, and Neil P. Strickland: \textsl{Elliptic spectra, the Witten
genus and the theorem of the cube}. Invent. Math. , 146(3):595$\textendash$687,
2001.

\bibitem{AHSsigma}Matthew Ando, Michael J. Hopkins, and Neil P. Strickland: \textsl{The sigma orientation
is an $H_{\infty}$ map}, Amer. J. Math., 126(2):247$\textendash$334, 2004.


\bibitem{ALRuan}Alejandro Adem, Johann Leida, Yongbin Ruan: \textsl{Orbifolds and stringy
topology}, Cambridge Tracts in Mathematics, 171. Cambridge
University Press, Cambridge, 2007.


\bibitem{KR}Atiyah, M.F.: \textsl{K-theory and reality}, Quart. J. Math. Oxford Ser. (2) 17 1966
367$\textendash$386.

\bibitem{Ati66}M.F.Atiyah: \textsl{Power operations in K-theory}, Quart. J. Math. Oxford Ser. (2) 17 1966 pp. 165$\textendash$193.

\bibitem{KC}Atiyah, M.F. and Segal, G.B.:
\textsl{Equivariant K-theory and completion}, J. Differential
Geometry 3 1969 1$\textendash$18.

\bibitem{BDBRT}Baas, Nils A.; Dundas, Bj$\phi$rn Ian; Rognes, John: \textsl{Two-vector bundles and forms of elliptic cohomology},
Topology, geometry and quantum field theory, 18$\textendash$45,
London Math. Soc. Lecture Note Ser., 308, Cambridge Univ. Press, Cambridge, 2004.

\bibitem{BDBRR}Baas, Nils A.; Dundas, Bj$\phi$rn Ian; Richter, Birgit; Rognes, John:
\textsl{Ring completion of rig categories},
J. Reine Angew. Math. 674 (2013), 43$\textendash$80.

\bibitem{SSCP}Tobias Barthel, Nathaniel Stapleton: \textsl{The character of the total power operation}, Geom. Topol. 21(2017), no.1, 385$\textendash$440.

\bibitem{BM}Clemens Berger, Ieke Moerdijk: \textsl{On an extension of the notion of Reedy
category}, Mathematische Zeitschrift, December 2011, Volume 269,
Issue 3, pp 977$\textendash$1004.

\bibitem{BG}Bohmann, Anna Marie: \textsl{Global orthogonal spectra},  Homology, Homotopy and Applications. Vol 16 (2014), No 1, 313$\textendash$332.






\bibitem{Dev96}Jorge A. Devoto: \textsl{Equivariant elliptic homology and finite
groups}, Michigan Math. J. , 43(1):3$\textendash$32, 1996.


\bibitem{FH}William Fulton, Joe Harris: \textsl{Representation Theory, a first course},  Springer GTM \textbf{129} 1991.

\bibitem{Gan06}Nora Ganter: \textsl{Orbifold genera, product
formulas and power operations}, Advances in Mathematics, 205(2006), no.1 84$\textendash$133.

\bibitem{Gan07}Nora Ganter: \textsl{Stringy power operations in Tate K-theory},
2007, available at arXiv: math/0701565.

\bibitem{Gan13}Nora Ganter: \textsl{Power operations in orbifold Tate
K-theory},  Homology Homotopy Appl. 15 (2013), no. 1, 313$\textendash$342.



\bibitem{Gepnerthesis}David Gepner: \textsl{Homotopy Topoi and Equivariant Elliptic
Cohomology}, Thesis (Ph.D.)每University of Illinois at Urbana-Champaign. 1999.

\bibitem{GKV}V. Ginzburg, M. Kapranov, E. Vasserot: \textsl{Elliptic algebras and equivariant elliptic cohomology I}, available at         arXiv:q-alg/9505012.



\bibitem{GP}J.P.C. Greenlees and J.P. May: \textsl{Localization and completion theorems for MU-
module spectra}, Ann. of Math. (2), 146(3) (1997), 509$\textendash$544.

\bibitem{Groj}I. Grojnowski: \textsl{Delocalised equivariant elliptic cohomology (1994)}, in Elliptic cohomology, volume 342 of London Math. Soc. Lecture
Note Ser., pages 114$\textendash$121. Cambridge Univ. Press, Cambridge, 2007.


\bibitem{HKR}Michael J. Hopkins, Nicolas J. Kuhn, Douglas C. Ravenel:  \textsl{Generalized group characters and complex
oriented cohomology theories}, J. Am. Math. Soc. 13 (2000) 553$\textendash$594.




\bibitem{HuanPower} Zhen Huan:  \textsl{Quasi-Elliptic Cohomology and its Power
Operations},
J. Homotopy Relat. Struct. (2018). \url{https://doi.org/10.1007/s40062-018-0201-y}.

\bibitem{HuanSpec} Zhen Huan: \textsl{Quasi-elliptic cohomology and its Spectrum},
available at arXiv:1703.06562.

\bibitem{Huanthesis}Zhen Huan: \textsl{Quasi-elliptic cohomology}, Thesis (Ph.D.)每University of Illinois at Urbana-Champaign. 2017. 290 pp. \url{http://hdl.handle.net/2142/97268}.

\bibitem{Huanuniv} Zhen Huan: \textsl{Universal Finite Subgroup of Tate Curve},
available at arXiv:1708.08637.

\bibitem{Huanlevel}Zhen Huan and Nathaniel Stapleton: \textsl{Level Structures, Loop
Spaces, and Morava E-theory}. Almost finished project.


\bibitem{Huanglobal} Zhen Huan: \textit{Almost Global Homotopy Theory}. In Preparation.
\bibitem{KM85}Nicholas M. Katz and Barry Mazur: \textsl{Arithmetic moduli of elliptic curves}, Annals of Mathematics Studies, vol. 108, Princeton University Press, Princeton, NJ,
1985.
\bibitem{JK}Joachim, Michael: \textsl{Higher coherences for equivariant
K-theory}, Structured ring spectra, 87$\textendash$114, London Math. Soc.
Lecture Note Ser., 315, Cambridge Univ. Press, Cambridge, 2004.
\bibitem{LProceeding}Landweber, P. S.: \textsl{Elliptic Curves and Modular Forms in Algebraic Topology: proceedings of a conference held at the Institute for Advanced Study}, Princeton, September 1986,
Lecture Notes in Mathematics.
\bibitem{LerStack}Eugene Lerman: \textsl{Orbifolds as stacks}, Enseign. Math. (2) 56 (2010), no. 3$\textendash$4, 315$\textendash$363.
\bibitem{Lurie}Jacob Lurie: \textsl{A Survey of Elliptic Cohomology}, in Algebraic Topology Abel Symposia Volume 4, 2009, pp 219$\textendash$277.

\bibitem{LUloop} E. Lupercio and B. Uribe: \textsl{Loop groupoids, gerbes, and twisted sectors on orbifolds}, in Orbifolds in Mathematics
and Physics, A. Adem J. Morava and Y. Ruan, eds., Contemporary
Math 310, AMS (2002), 163$\textendash$184.


\bibitem{LU01} Ernesto Lupercio and Bernardo Uribe: \textsl{Inertia orbifolds, configuration spaces and the ghost loop
space}, Quarterly Journal of Mathematics 55, Issue 2, pp. 185$\textendash$201.



\bibitem{MM}M.A.Mandell, J.P.May: \textsl{Equivariant orthogonal spectra and
S-modules}, Mem.,Amer. Math. Soc. 159 (2002), no. 755, x+108 pp.



\bibitem{MMSS}M.A.Mandell, J.P.May, S.Schwede, and B.Shipley:
\textsl{Model categories of diagram spectra}, Proc. London Math.
Soc. 82(2001), 441$\textendash$512.

\bibitem{MV}P. Mitter, C. Viallet:  \textsl{On the bundle of connections and the
gauge orbit manifold in Yang$\textendash$Mills theory}, Commun. Math. Phys.
79 (1981) 457$\textendash$472.


\bibitem{Moe02}I. Moerdijk: \textsl{Orbifolds as groupoids: an introduction}, Orbifolds in mathematics and
physics, Madison, WI, 2001, Contemparary Mathematics, vol. 310,
American Mathematics Society, Providence, RI, 2002, pp. 205$\textendash$222.






\bibitem{Rez11}Charles Rezk: \textsl{Quasi-elliptic cohomology}, unpublished manuscript, 2011.

\bibitem{Rez14} Charles Rezk: \textsl{Global Homotopy Theory and Cohesion}, unpublished manuscript, 2014.



\bibitem{Rez16}Charles Rezk: \textsl{Loopy Stuff}, 2016.


\bibitem{SSST}Tomer M. Schlank, Nathaniel Stapleton: \textsl{A transchromatic proof of Strickland's theorem}, Adv. Math. 285 (2015), 1415$\textendash$1447.
\bibitem{SegalequiK}Segal, Graeme: \textsl{Equivariant K-theory}, Inst. Hautes \'{E}tudes Sci. Publ. Math. No. 34 1968 129$\textendash$151.


\bibitem{SS}Stefan Schwede: \textsl{Global Homotopy Theory},
v0.23/April 30, 2015, Preliminary and incomplete version, \url{http://www.math.uni-bonn.de/people/schwede/global.pdf}.
\bibitem{STGCM}Nathaniel Stapleton:\textsl{Transchromatic generalized character maps},  Algebr. Geom. Topol. 13 (2013), no. 1, 171$\textendash$203.
\bibitem{STTCM}Nathaniel Stapleton:\textsl{Transchromatic twisted character maps},  J. Homotopy Relat. Struct. 10 (2015), no. 1, 29$\textendash$61.
\bibitem{Str98}N. P. Strickland: \textsl{Morava E-theory of symmetric groups}, Topology 37 (1998), no. 4, 757$\textendash$779.

\bibitem{Yau10}Donald Yau: \textsl{Lambda-rings}, World Scientific Publishing Co. Pte. Ltd., Hackensack, NJ, 2010.
\end{thebibliography}
\end{document}